\documentclass{amsart}

\usepackage{amsmath, amssymb, amsthm, amscd, amssymb, color, latexsym, tikz, hyperref}
\usepackage[margin=1in]{geometry}
\usetikzlibrary{calc, external, fit, shapes}

\numberwithin{equation}{section}
\newtheorem{theorem}{Theorem}[section]
\newtheorem{proposition}[theorem]{Proposition}
\newtheorem{corollary}[theorem]{Corollary}
\newtheorem{lemma}[theorem]{Lemma}

\newtheorem*{theorem*}{Theorem}
\newtheorem*{proposition*}{Proposition}
\newtheorem*{corollary*}{Corollary}
\newtheorem*{lemma*}{Lemma}
\newtheorem*{question*}{Question}
\newtheorem*{claim*}{Claim}
\newtheorem*{conjecture*}{Conjecture}

\theoremstyle{definition}
\newtheorem{definition}[theorem]{Definition}
\newtheorem{example}[theorem]{Example}

\newtheorem*{definition*}{Definition}
\newtheorem*{example*}{Example}
\newtheorem*{remark*}{Remark}
\newtheorem*{problem*}{Problem}
\newtheorem*{assumption*}{Assumption}

\title[Asymptotics of Brownian occupation measures with unusually large intersections]{Asymptotics of Brownian occupation measures \\ with unusually large intersections}

\author{Jiyun Park}
\address{Department of Mathematics, Stanford University, USA. \href{mailto:jiyunp@stanford.edu}{\tt jiyunp@stanford.edu}}

\begin{document}

\begin{abstract}
    We prove that the occupation measures of Brownian motions conditioned to have large intersections converge weakly, up to spatial shifts, to the measure whose density is the square of an optimizer of the Gagliardo-Nirenberg inequality. We do so by proving a large deviation principle (LDP) for Brownian occupation measures conditioned either on large self-intersections or large mutual intersections. To this end, we derive a compact LDP for unconditioned Brownian occupation measures, generalizing the work of Mukherjee and Varadhan \cite{MukherjeeVaradhan2016}. We also prove the LDP for Brownian occupation measures tilted by their intersections in the same topology. A key tool of independent interest is an exponentially good approximation of the intersection measure tested against all bounded measurable functions, from which we further get the LDP for the intersection measure of independent Brownian motions.
\end{abstract}

\maketitle

\section{Introduction}

\subsection{Intersections of Brownian motions}

Given a Brownian motion $W_t$, it is very natural to ask how much its path intersects itself. This is measured by the $q$-fold \emph{self-intersection local time}\footnote{This is only well-defined in $\mathbb{R}$, as the integral blows up to infinity in higher dimensions.}, formally defined\footnote{For convenience, we often write formal statements involving delta functions. All such statements can be made rigorous by replacing the delta functions with a sequence of mollifiers and taking limits. We omit these details, as they are now standard techniques in the literature (e.g., see Le Gall's moment identity \cite{LeGall1994} or the constructions of intersection local times in \cite{Chen2010}).} by
\[
\beta([0, t]^q) := \int_{\mathbb{R}} \left( \int_0^t \delta (W_s - x) \mathrm{d} s \right)^q \mathrm{d} x
\]
for any $q > 1$ (where $\delta$ is the Dirac delta measure). Similarly, we define the \emph{mutual intersection local time}\footnote{It is known that $\alpha([0, t]^p)$ is positive and finite when $d(p-1) < 2p$. On the other hand, it is zero when $d(p-1) \ge 2p$.} for $p$ independent Brownian motions $W^1 , \dots , W^p$ in $\mathbb{R}^d$, now for any $d \ge 1$ and $p \ge 2$ such that $d(p-1) < 2p$.
\[  
\alpha([0, t]^p) := \int_{\mathbb{R^d}} \int_{[0, t]^p} \delta(W^1(s_1) - x) \delta(W^2(s_2) - x) \dots \delta(W^p(s_p) - x) \mathrm{d} s_1 \dots \mathrm{d} s_p  \mathrm{d} x.
\]
Early works on intersections of Brownian paths date back to Dvoretzsky, Erd\"{o}s, Kakutani, and Taylor in the 1950s \cite{DvoretzkyErdoesKakutani1950} and have been studied extensively since (see \cite{Lawler2013, FreidlinLeGall1992} for a survey).

In particular, the monograph by Chen \cite{Chen2010} provides comprehensive information on the \emph{upper tail large deviations} of intersection local times. Among his main results are the following.
\begin{equation}
\label{eq:self-var-problem}
\lim_{t \to \infty} \frac{1}{t} \log \mathbb{P}( \beta([0, t]^q) \ge t^q) = -\inf_{\substack{\psi \in H^1(\mathbb{R}) \\ \| \psi \|_2 =1}} \left\{ \frac{1}{2} \| \nabla \psi \|_2^2 : \| \psi \|_{2q} = 1 \right\} =: -\Theta_{1, q},
\end{equation}
and when $p \ge 2$ and $d(p-1) < 2p$,
\begin{equation}
\label{eq:mutual-var-problem}
\lim_{t \to \infty} \frac{1}{t} \log \mathbb{P}( \alpha([0, t]^p) \ge t^p) = -\inf_{\substack{\psi^j \in H^1(\mathbb{R}^d) \\ \| \psi^j \|_2 =1}} \bigg\{ \frac{1}{2} \sum_{j=1}^p \| \nabla \psi^j \|_2^2 : \Big\| \prod_{j=1}^p \psi^j \Big\|_2 = 1 \bigg\} = -p \cdot \Theta_{d, p}.
\end{equation}
We remark that \eqref{eq:mutual-var-problem} is optimized when $\psi^1 = \dots = \psi^p$, which explains the repetition of the constant $\Theta_{d, p}$.

The solutions to the optimization problems are also known to be unique up to spatial translations. In fact, they are precisely the functions that achieve equality in the Gagliardo-Nirenberg inequality,
\begin{equation}
\label{eq:GN-inequality}
\|  \psi \|_{2q} \le \kappa_{d, q} \| \nabla  \psi \|_{2}^{\frac{d(q-1)}{2q}} \| \psi \|_2^{1 - \frac{d(q-1)}{2q}} \quad \text{for all }\psi \in H^1(\mathbb{R}^d).
\end{equation}

Given~\eqref{eq:self-var-problem}--\eqref{eq:GN-inequality}, it is natural to expect that the path of the Brownian motion(s) conditioned on $\{\beta([0, t]^q) \ge t^q \}$ or $\{\alpha([0, t]^p) \ge t^p \}$ relate to solutions of the optimization problems. In this context, we consider Brownian occupation measure
\[  
L_t(A) = \frac{1}{t} \int_0^t \mathbf{1}_A(W_s) \mathrm{d} s
\]
conditioned on the event $\{ \beta([0, t]^q) \ge \lambda t^q \}$. We prove the convergence of $L_t(\cdot)$ in the weak topology \emph{up to spatial shifts}, i.e., in the topology $\widetilde{\mathcal{M}}_1(\mathbb{R}) = \mathcal{M}_1(\mathbb{R}) / \sim$ where $\mathcal{M}_1(\mathbb{R})$ is the set of probability measures equipped with the weak topology and the equivalence relation $\mu \sim \nu$ (denoted by $\widetilde{\mu} = \widetilde{\nu}$) implies $\mu = \nu(\cdot - x)$ for some $x \in \mathbb{R}$. Similar definitions may be made for sub-probability measures $\widetilde{\mathcal{M}}_{\le 1}(\mathbb{R})$ or finite measures $\widetilde{\mathcal{M}}(\mathbb{R})$.

\begin{theorem}
    \label{thm:self-cond-conv}
    Conditioned on $\{ \beta([0, t]^q) \ge t^q \}$,
    \[  
    \lim_{t \to \infty} \widetilde{L}_t = \widetilde{\mu}_{1, q} \quad \text{in } \widetilde{\mathcal{M}}_1(\mathbb{R}),
    \]
    where $\mu_{1, q}$ has density $\psi_{1, q}^2$ and $\psi_{1, q}$ uniquely solves \eqref{eq:self-var-problem} (up to spatial shifts).
\end{theorem}
We obtain similar results for the $p$-fold mutual intersections, where we now quotient under \emph{diagonal} spatial shifts. That is, two tuples of measures $\mu^{\otimes p} = (\mu^1, \dots , \mu^p)$ and $\nu^{\otimes p} = (\nu^1, \dots , \nu^p)$ in $(\mathcal{M}_1(\mathbb{R}^d))^p$ are equivalent (denoted by $\widetilde{\mu}^{\otimes p} = \widetilde{\nu}^{\otimes p}$) if there exists $x \in \mathbb{R}^d$ such that $\mu^j (\cdot) = \nu^j(\cdot - x)$ for all $j = 1, 2, \dots , p$. We denote this quotient space as $\widetilde{\mathcal{M}}_1^{\otimes p}(\mathbb{R}^d)$. This is \emph{not} the same as $(\widetilde{\mathcal{M}}_1(\mathbb{R}^d))^p$, since we are only allowing diagonal shifts.

\begin{theorem}
    \label{thm:mutual-cond-conv}
    Suppose $p \ge 2$ and $d(p-1) < 2p$. Conditioned on $\{ \alpha([0, t]^p) \ge t^p \}$, the tuple of occupation measures $L_t^{\otimes p} = (L_t^1 , \dots , L_t^p)$ satisfies
    \[  
    \lim_{t \to \infty} \widetilde{L}_t^{\otimes p} = \widetilde{\mu}_{d, p}^{\otimes p} \quad \text{in } \widetilde{\mathcal{M}}_1^{\otimes p}(\mathbb{R}),
    \]
    where $\mu^{\otimes p}_{d, p} = (\mu_{d, p}^1 , \dots , \mu_{d, p}^p)$. Each $\mu_{d, p}^j$ has density $\psi_{d, p}^2$ where $(\psi_{d, p} , \dots , \psi_{d, p})$ uniquely solves \eqref{eq:mutual-var-problem} (up to diagonal shifts).
\end{theorem}

Due to the Brownian scaling $ \{W_{cs}  : 1 \le s \le t  \}\overset{d}{=} \{\sqrt{c}W_s : 1 \le s \le t \}$, we may extend our results to deviations of any order (with exponential tail decay) via the scaling property
\[  
\alpha([0, ct]^p) \overset{d}{=} c^{\frac{2p - d(p-1)}{2}} \alpha([0, t]^p), \quad \beta([0, ct]^q) \overset{d}{=} c^{\frac{q + 1}{2}} \beta([0, t]^q)
\]
of \cite[Propositions~2.2.6, 2.3.3]{Chen2010}. For example, Theorem~\ref{thm:self-cond-conv} implies the following statements.
\begin{enumerate}
    \item Given $\beta([0,t]^2) \ge \lambda t^2$, $\widetilde{L}_t$ converges to the distribution with density $\lambda \psi_{1, q}^2 (\lambda x)$ in $\widetilde{\mathcal{M}}_1(\mathbb{R})$.
    This comes from taking $c = \lambda^2$.
    \item Given $\beta[0, t]^2 \ge t^3$, the rescaled measure $L_t'(A) := L_t (t^{-1} A)$ converges to $\widetilde{\mu}_{1, 2}$ in $\widetilde{\mathcal{M}}_1(\mathbb{R})$. This comes from taking $c = t^{2}$.
\end{enumerate}

Apart from intrinsic interest, intersections of Brownian motions are a prototypical example of functionals over the path measure. Other models include the volume of the Wiener sausage \cite{BergBolthausenHollander2001, BergBolthausenHollander2004}, the intersections and volume of a random walk \cite{JainOrey1968, Chen2010, AsselahSchapira2021, AsselahSchapira2023}, and the capacity of random walk ranges \cite{Schapira2020, AsselahSchapiraSousi2019, AsselahSchapira2020, AdhikariOkada2023, DemboOkada2024, AsselahSchapiraSousi2018, AdhikariPark2025}. It is also possible to consider other Markov processes and potentials \cite{BassChenRosen2009}. These models share many similar features and methods used to solve one model can often be transferred to other systems. In particular, our paper draws inspiration from prior works which prove weak convergence of random walks with small volume \cite{PoisatErhard2023} and Brownian motions under the Coulomb potential \cite{MukherjeeVaradhan2016, Mukherjee2017, KoenigMukherjee2017, BolthausenKoenigMukherjee2017} or the polaron measure \cite{MukherjeeVaradhan[2022]copyright2022}. We also believe the additional tools developed in this paper may be generalized and applied to other problems of a similar nature. Indeed, the two techniques we develop here, namely the LDP and exponential approximations of the Brownian occupation measures, are quite general both in proof strategy and result.

Another motivation for Theorems~\ref{thm:self-cond-conv} and \ref{thm:mutual-cond-conv} is that these conditional distributions are often closely related to Gibbs measures created by tilting the probability measure to favor large intersections. This is a widely studied model in statistical physics and used to model self-attracting polymers (e.g., see \cite{Hollander2009}). In this context, we show that Brownian occupation measures under the Gibbs measure also converge to rescaled versions of the limits of Theorems~\ref{thm:self-cond-conv} and \ref{thm:mutual-cond-conv}.

\begin{theorem}
\label{thm:self-tilted-conv}
    For any $q > 1$ and $0 < \gamma < \frac{2q}{q-1}$, consider the Gibbs measure
    \begin{equation}
    \label{eq:self-gibbs-measure}
        \mathrm{d} \widehat{\mathbb{P}}_{t} = \frac{1}{Z_t} \exp \left\{ t^{1 - \gamma} \beta([0, t]^q)^{\gamma/q} \right\} \mathrm{d} \mathbb{P}.
    \end{equation}
    Then, the Brownian occupation measures $\widetilde{L}_t$ under the law $\widehat{\mathbb{P}}_t \circ \widetilde{L}_t^{-1}$ converge to
    \[  
    \lim_{t \to \infty} \widetilde{L}_t = \widetilde{\mu}_{1, q, \gamma} \quad \text{in } \widetilde{\mathcal{M}}_1(\mathbb{R}),
    \]
    where $\mu_{1, q, \gamma}$ has density $\psi_{1, q, \gamma}^2$ and $\psi_{1, q, \gamma}$ is the unique (up to shifts) solution to the variational problem
    \begin{equation}
        \label{eq:self-tilted-var}
    \rho_{1, q, \gamma} = \sup_{\substack{\psi \in H^1(\mathbb{R}) \\ \| \psi \|_2 = 1}} \left\{\| \psi \|_{2q}^{2\gamma} - \frac{1}{2} \| \nabla \psi \|_2^2 \right\}.
    \end{equation}
\end{theorem}

When $\gamma > \frac{2q}{q-1}$, the ground state energy $\frac{1}{t} \log Z_t$ diverges and $\widetilde{L}_t$ does not converge. The most studied case is when $\gamma = 1$, which corresponds to tilting by $\beta([0, t]^q)^{1/q}$ with no additional time factor. Another common choice is to take $q = \gamma = 2$, in which case we get 
\[
\frac{1}{t} \beta([0, t]^2) = \frac{1}{t} \int_{\mathbb{R}} \int_0^t \int_0^t \delta(W_r - x) \delta(W_s - x)  \mathrm{d} r \mathrm{d} s \mathrm{d} x= \frac{1}{t} \int_0^t \int_0^t \delta(W_r - W_s)  \mathrm{d} r \mathrm{d} s.
\]

\begin{theorem}
\label{thm:mutual-tilted-conv}
    For any $d \ge 1$, $p \ge 2$ such that $d(p-1) < 2p$ and $0 < \gamma < \frac{2p}{d(p-1)}$, take the Gibbs measure
    \begin{equation}
    \label{eq:mutual-gibbs-measure}
    \mathrm{d} \widehat{\mathbb{P}}_t^{\otimes p} = \frac{1}{Z_t} \exp \left\{ pt^{1-\gamma} \left( \alpha([0, t]^p) \right)^{\gamma/p} \right\} \mathrm{d} \mathbb{P}^{\otimes p}.
    \end{equation}
    Then, the Brownian occupation measures $\widetilde{L}_t^{\otimes p}$ under the law $\widehat{\mathbb{P}}_t^{\otimes p} \circ (\widetilde{L}_t^{\otimes p})^{-1}$ converge to
    \[  
    \lim_{t \to \infty} \widetilde{L}_t^{\otimes p} = \widetilde{\mu}_{d, p, \gamma}^{\otimes p} \quad \text{in } \widetilde{\mathcal{M}}_1^{\otimes p}(\mathbb{R}^d).
    \]
    where $\mu_{d, p, \gamma}^{\otimes p} = (\mu_{d, p,\gamma}^1, \dots ,\mu_{d, p,\gamma}^p)$. Each $\mu_{d, p,\gamma}^j$ has density $\psi_{d, p,\gamma}^2$, where $(\psi_{d, p,\gamma} , \dots , \psi_{d, p,\gamma})$ is the unique (up to diagonal shifts) solution to the variational problem
    \begin{equation}
    \label{eq:mutual-tilted-var}
    p \cdot \rho_{d, p, \gamma} = \sup_{\substack{\psi \in (H^1(\mathbb{R}^d))^p \\ \| \psi^j \|_2 = 1}} \left\{ p \Big\| \prod_{j=1}^p \psi^j \Big\|_{2}^{2\gamma/p} - \frac{1}{2} \sum_{j=1}^p \| \nabla \psi^j \|_2^2\right\}.
    \end{equation}
\end{theorem}
Similarly to \eqref{eq:mutual-var-problem}, equation \eqref{eq:mutual-tilted-var} is maximized when $\psi^1 = \dots = \psi^p$ and hence reduces to \eqref{eq:self-tilted-var}.

When proving the above theorems, our starting point is the celebrated Donsker-Varadhan weak LDP \cite{DonskerVaradhan1975, DonskerVaradhan1976, DonskerVaradhan1983},
\begin{equation}
\label{eq:DV-LDP}
\mathbb{P} (L_t \approx \mu) = \exp \left\{ -\frac{t}{2} \left\| \nabla \sqrt{\frac{d\mu}{\mathrm{d} x}} \right\|_2^2 + o(t) \right\}.
\end{equation}
From a naive perspective, Theorems~\ref{thm:self-cond-conv}--\ref{thm:mutual-tilted-conv} seem to be mere applications of the contraction principle. The main challenge in proving our results is bridging the two following major gaps in this rationale.

Firstly, the intersection local times $\beta([0, t]^q)$ and $\alpha([0, t]^p)$ are not continuous functionals of the occupation measures. To overcome this problem, one must define continuous analogs of these quantities and show that are exponentially good approximations of the true values. We go even further and show that the occupation measures themselves are well-approximated by their smoothed versions (see Section~\ref{subsec:intro-exp-approx} for a precise statement). Our methods unify and generalize several previous attempts \cite{Chen2010, KoenigMoerters2002, KoenigMoerters2006, KoenigMukherjee2013, Mukherjee2017, Mori2023} using a novel (and purely probabilistic) strategy which is quite general. This exponential approximation is the main technical challenge of the paper, and we believe our approach and result may have applications to settings beyond what is considered here---see Section~\ref{subsec:intro-exp-approx} and Section~\ref{sec:exp-approx} for more details.

Even after overcoming the lack of continuity, the remaining obstacle is that \eqref{eq:DV-LDP} is only a \emph{weak} LDP, the key problem being that $\mathcal{M}_1(\mathbb{R}^d)$ is not compact. We resolve this by modifying the topology on the space of occupation measures, generalizing the approach of Mukherjee and Varadhan \cite{MukherjeeVaradhan2016,Mukherjee2017} with a key change in how we generalize to joint distributions. This gives us a compact topology on which the  Brownian occupation measures satisfy a full LDP. The topology and associated LDP is discussed in more detail in Section~\ref{subsec:intro-LDP} and Section~\ref{sec:MV-top}.

\subsection{LDP for Brownian occupation measures}
\label{subsec:intro-LDP}

A critical limitation of the Donsker-Varadhan weak LDP is that there is no upper bound for closed sets. To bypass this obstruction, previous results on intersection local times often used methods such as simply analyzing the Brownian motion on a bounded region \cite{KoenigMoerters2002, KoenigMoerters2006, KoenigMukherjee2013} or folding the Brownian paths into a large torus \cite{Chen2010}. Another approach is to compare the Brownian motion with the Ornstein-Uhlenbeck process \cite{DonskerVaradhan1983} which, unlike the Brownian motion, is exponentially tight. However, while these techniques work well for the values of the intersection local times, they cannot handle questions about the underlying occupation measures.

To this end, Mukherjee and Varadhan \cite{MukherjeeVaradhan2016} derived a full LDP by introducing a new topology on the space of occupation measures. They do so by first taking the quotient space $\widetilde{\mathcal{M}}_1(\mathbb{R}^d)$ of orbits under spatial translations, and then considering (countable) combinations of such orbits. Generalized to $p$-fold products, this leads to the set
\begin{equation} 
\label{eq:MV-bdd-set}
\widetilde{\mathcal{X}}^{\otimes p}_{\le 1} (\mathbb{R}^d) = \left\{ \xi^{\otimes p} = \{ \widetilde{\alpha}_i^{\otimes p} \}_{i \in I}: \widetilde{\alpha}_i^{\otimes p} \in \widetilde{\mathcal{M}}_{\le 1}^{\otimes p}(\mathbb{R}^d), \; \sum_{i \in I} \alpha_i^j(\mathbb{R}^d) \le 1\right\}.
\end{equation}
Equipped with a suitable metric $\mathbf{D}^{\otimes p}$ (to be defined in Section~\ref{sec:MV-top}), the space $\widetilde{\mathcal{X}}^{\otimes p}_{\le 1}(\mathbb{R}^d) $ is compact and contains $\widetilde{\mathcal{M}}_1^{\otimes p} (\mathbb{R}^d)$ as a dense subspace. We can then prove a full LDP for $\widetilde{L}_t^{\otimes p}$ in $\widetilde{\mathcal{X}}_{\le 1}^{\otimes p}$:
\begin{proposition}
\label{prop:brownian-ldp}
    The distributions $\widetilde{L}_t^{\otimes p}$ satisfy a large deviation principle in the compact metric space $(\widetilde{\mathcal{X}}_{\le 1}^{\otimes p}, \mathbf{D}^{\otimes p})$ with good rate function
    \begin{equation}
    \label{eq:rate-function}
    \mathcal{I}(\xi^{\otimes p}) = \begin{cases}
    \displaystyle \frac{1}{2} \sum_{i \in I} \sum_{j=1}^p \| \nabla \psi_i^j \|_2^2  & \text{if } \psi_i^j = \sqrt{\frac{\mathrm{d} \alpha_i^j}{\mathrm{d} x}} \in H^1(\mathbb{R}^d) \text{ for all } i, j \\
    \infty &\text{otherwise}.
    \end{cases}
\end{equation}
\end{proposition}

The case $p = 1$ was done in \cite{MukherjeeVaradhan2016} and has been used in great success to show convergence of Brownian motions under the Coulomb potential or the polaron measure \cite{MukherjeeVaradhan2016, KoenigMukherjee2017, BolthausenKoenigMukherjee2017, MukherjeeVaradhan[2022]copyright2022}. For joint distributions, \cite{Mukherjee2017} shows a similar theorem when $p = 2$ and there are also variations of this LDP for random walks \cite{BatesChatterjee2020, ErhardFrancoJesusSantana2025, ErhardPoisat2026}. The idea of using shift-invariance to construct compactifications of was introduced by Parthasarathy, Ranga Rao, and Varadhan \cite{ParthasarathyRangaRaoVaradhan1962} and later developed into a general theory of ``concentration compactness'' by Lions \cite{Lions1984,Lions1984a} with numerous applications in PDEs and calculus of variations.

However, the topology used in \cite{Mukherjee2017, ErhardPoisat2026} for joint measures is slightly different from the one we consider here. We use an alternate definition of $\widetilde{\mathcal{X}}_{\le 1}^{\otimes p}$ which we feel is a more natural generalization of \cite{MukherjeeVaradhan2016} that preserves full information of the marginals. Indeed, a major drawback of \cite{Mukherjee2017, ErhardPoisat2026} is that the maps to the marginals $\widetilde{L}_t^{\otimes p} \mapsto \widetilde{L}_t^j$ are not continuous; our construction resolves this problem.

Equipped with this new LDP (and after overcoming the lack of continuity), we apply tools from large deviation theory to obtain the following LDP for occupation measures conditioned on large intersections.
\begin{theorem}
    \label{thm:self-cond-ldp}
    Conditioned on $\{\beta([0,t]^q) \ge t^q\}$, $\widetilde{L}_t$ satisfies an LDP in the compact metric space $\widetilde{\mathcal{X}}_{\le 1}(\mathbb{R})$ with good rate function 
    \[  
    \mathcal{I}_{1, q}^{\mathrm{cond}}(\xi) = \begin{cases}
        \displaystyle \mathcal{I}(\xi) - \Theta_{1, q}  & \text{if } \mathcal{I}(\xi)< \infty \text{ and } \sum_{i \in I}\| \psi_i \|_{2q}^{2q} \ge 1\\
    \infty &\text{otherwise},
    \end{cases}
    \]
    where $\mathcal{I}(\cdot)$ and $\Theta_{1, q}$ are as in \eqref{eq:rate-function} and\eqref{eq:self-var-problem}, respectively.
\end{theorem}

\begin{theorem}
    \label{thm:mutual-cond-ldp}
    Suppose $p \ge 2$ and $d(p-1) < 2p$. Conditioned on the event $\{\alpha([0,t]^p) \ge t^p\}$, $\widetilde{L}_t^{\otimes p}$ satisfies an LDP in the compact metric space $\widetilde{\mathcal{X}}_{\le 1}^{\otimes p}(\mathbb{R}^d)$ with good rate function
    \[  
    \mathcal{I}_{d, p}^{\mathrm{cond}}(\xi^{\otimes p}) = \begin{cases}
        \displaystyle \mathcal{I}(\xi^{\otimes p}) - p \cdot \Theta_{d, p}  & \text{if }\mathcal{I}(\xi^{\otimes p})< \infty \text{ and } \sum_{i \in I}\| \prod_{j = 1}^p \psi_i^j \|_{2}^{2} \ge 1\\
    \infty &\text{otherwise},
    \end{cases}
    \]
    where $\mathcal{I}(\cdot)$ and $\Theta_{d, p}$ are as in \eqref{eq:rate-function} and \eqref{eq:mutual-var-problem}, respectively.   
\end{theorem}

Theorems~\ref{thm:self-cond-conv} and \ref{thm:mutual-cond-conv} are immediate consequences of Theorems~\ref{thm:self-cond-ldp}, \ref{thm:mutual-cond-ldp} and Lemma~\ref{lem:self-var-problem}, which shows that $\widetilde{\mu}_{1, q}$ and $\widetilde{\mu}_{d, p}^{\otimes p}$ are the unique minimizers of the respective rate functions. Note that the LDP is in the topology of $\widetilde{\mathcal{X}}_{\le 1}^{\otimes p}$, while the convergence in Theorems~\ref{thm:self-cond-conv} and \ref{thm:mutual-cond-conv} are in the weak topology. This is possible because $\widetilde{\mathcal{X}}_{\le 1}^{\otimes p}$ contains $\widetilde{\mathcal{M}}_1^{\otimes p}$ as a subspace---since both the sequence $\widetilde{L}_t^{\otimes p}$ and the limiting measure $\widetilde{\mu}_{d, p}$ lie in $\widetilde{\mathcal{M}}_1^{\otimes p}$, convergence in $\widetilde{\mathcal{X}}^{\otimes p}$ implies convergence in $\widetilde{\mathcal{M}}_1^{\otimes p}$.

Similarly, we also derive the LDP for the Gibbs measures introduced in Theorems~\ref{thm:self-tilted-conv} and \ref{thm:mutual-tilted-conv}. As in the preceding case of the conditional measure, these LDPs (combined with Lemma~\ref{lem:self-tilt-var}) immediately imply Theorems~\ref{thm:self-tilted-conv} and \ref{thm:mutual-tilted-conv}.

\begin{theorem}
    \label{thm:self-tilted-ldp}
    For any $0 <\gamma < \frac{2q}{q-1}$, the distributions $\widehat{\mathbb{P}}_t \circ (\widetilde{L}_t)^{-1}$ under the Gibbs measure \eqref{eq:self-gibbs-measure} satisfy an LDP in $\widetilde{\mathcal{X}}_{\le 1}(\mathbb{R})$ with good rate function
    \[
    \mathcal{I}_{1, q, \gamma}^{\mathrm{Gibbs}}(\xi) = \begin{cases}
    \mathcal{I}(\xi) - \left( \sum_{i \in I} \| \psi_i \|_{2q}^{2q} \right)^{\gamma / q}  + \rho_{1, q, \gamma} & \text{if } \mathcal{I}(\xi) < \infty\\
    \infty & \text{otherwise},
    \end{cases}
    \]
    where $\mathcal{I}(\cdot)$ and $\rho_{1, q, \gamma}$ are defined in \eqref{eq:rate-function} and \eqref{eq:self-tilted-var}, respectively.
    
\end{theorem}
\begin{theorem}
    \label{thm:mutual-tilted-ldp}
    For any $d \ge 1$, $p \ge 2$ such that $d(p-1) < 2p$ and any $0 < \gamma < \frac{2p}{d(p-1)}$, the distribution $\widehat{\mathbb{P}}_t^{\otimes p} \circ (\widetilde{L}_t^{\otimes p})^{-1}$ under the Gibbs measure \eqref{eq:mutual-gibbs-measure} satisfies an LDP in $\widetilde{\mathcal{X}}^{\otimes p}(\mathbb{R}^d)$ with good rate function
    \[
    \mathcal{I}_{d, p, \gamma}^{\mathrm{Gibbs}}(\xi^{\otimes p}) = \begin{cases}
    \displaystyle \mathcal{I}(\xi^{\otimes p}) - p\biggl(\sum_{i \in I} \Big\| \prod_{j=1}^p \psi_i^j \Big\|_2^2 \biggr)^{\gamma/p} + p \cdot \rho_{d, p, \gamma} & \text{if } \mathcal{I}(\xi^{\otimes p}) < \infty\\
    \infty & \text{otherwise},
    \end{cases}
    \]
    where $\mathcal{I}(\cdot)$ and $\rho_{d, p, \gamma}$ are defined in \eqref{eq:rate-function} and \eqref{eq:mutual-tilted-var}, respectively.
\end{theorem}

\subsection{Exponential approximations of intersection measures}
\label{subsec:intro-exp-approx}

Once we have an LDP for the occupation measures, the remaining problem is that the intersection local times are not continuous functionals of the occupation measures. To overcome this obstacle, we proceed via an exponential approximation. That is, we define continuous analogs $\alpha_\epsilon([0, t]^p)$, $\beta_{\epsilon}([0,t]^q)$ of the intersection local times and show that they are good approximations in the sense that for any $\lambda > 0$,
\begin{align}
\label{eq:mutual-exp-approx}
\limsup_{\epsilon \to 0} \lim\sup_{t \to \infty} \frac{1}{t} \log \mathbb{E} \exp \left\{ \lambda \left| \alpha([0, t]^p) - \alpha_{\epsilon}([0, t]^p) \right|^{1/p} \right\} &= 0,\\
\label{eq:self-exp-approx}
\limsup_{\epsilon \to 0} \lim\sup_{t \to \infty} \frac{1}{t} \log \mathbb{E} \exp \left\{ \lambda \left| \beta([0, t]^q) - \beta_{\epsilon}([0, t]^q) \right|^{1/q} \right\} &= 0.
\end{align}

In reality, we go a great deal further and show that the occupation measures themselves are well-approximated. Recall that when $d = 1$, the self-intersection $\beta([0, t]^q)$ is equal to $\| \ell_t \|_q^q$, where $\ell_t$ is the density of the (pre-normalized) occupation measure,
\[  
\ell_t(x) := \int_0^t \delta(W_s - x) \mathrm{d} s.
\]
We approximate $\ell_t$ by convolving it with the Gaussian kernel $p_{\epsilon}$,
\[  
\ell_{t, \epsilon}(x) := \int_{-\infty}^{\infty} p_{\epsilon}(x -y) \ell_t(y) \mathrm{d} y = \int_{-\infty}^{\infty} \int_0^t  p_{\epsilon}(x -y) \delta(W_s - y)  \mathrm{d} s \mathrm{d} y = \int_0^t p_{\epsilon}(W_s - x) \mathrm{d} s
\]
and define $\beta_{\epsilon}([0, t]^q) := \| \ell_{t, \epsilon} \|_q^q$. We show that $\ell_{t, \epsilon}$ is an exponentially good approximation of $\ell_t$ in $L^q(\mathbb{R})$ for all $q > 1$.
\begin{proposition}
    \label{prop:self-exp-approx}
    For any $\lambda > 0$,
    \[  
    \limsup_{\epsilon \to 0} \limsup_{t \to \infty} \frac{1}{t} \log \mathbb{E} \exp \{ \lambda \| \ell_t - \ell_{t, \epsilon} \|_q \} = 0.
    \]
\end{proposition}
This immediately implies \eqref{eq:self-exp-approx}, since
\[  
|\beta([0, t]^q)^{1/q} - \beta_{\epsilon} ([0, t]^q)^{1/q}| = | \| \ell_t \|_q - \| \ell_{t, \epsilon} \|_q | \le \| \ell_t - \ell_{t, \epsilon} \|_q.
\]

Now we move to the mutual intersection case. For $p$ independent Brownian motions, we consider the \emph{intersection measure} $\ell_t^{\otimes p}$ (formally) defined as
\[
\ell_t^{\otimes p}(A) = \int_A \int_{[0, t]^p}  \prod_{j=1}^p \delta(W^j(s_j) - y) \mathrm{d}\mathbf{s} \mathrm{d} y
\]
and its smoothed approximation
\[
\ell_{t, \epsilon}^{\otimes p}(A) = \int_A \int_{[0, t]^p} \prod_{j=1}^p p_{\epsilon}(W^j(s_j) - y) \mathrm{d}\mathbf{s} \mathrm{d} y.
\]
This measures the amount of time the Brownian motions spend within a region\footnote{The measure $\ell_t^{\otimes p}$ also often goes by the name \emph{intersection local time}, but we reserve that name for the quantities $\alpha([0, t]^p)$ and $\beta([0, t]^q)$ in this paper. Instead, we shall always refer to $\ell_t^{\otimes p}$ as the \emph{intersection measure}.}\footnote{We also remark that while $t^{-1} \ell_t$ is the density of $L_t$ (and so we often use them interchangeably), $t^{-p} \ell_t^{\otimes p}$and $L_t^{\otimes p}$ are \emph{not} the same object. Indeed, $\ell_t^{\otimes p}$ is a measure on $\mathbb{R}^d$ while $L_t^{\otimes p}$ is a tuple of $p$ measures.}. From this perspective, $\alpha([0, t]^p)$ is simply the total mass of the intersection measure, $\ell_t^{\otimes p}(\mathbb{R}^d)$. We remark that while $\ell_{t, \epsilon}^{\otimes p}$ has density $\ell_{t, \epsilon}^{\otimes p}(d x) = \prod_{j=1}^p \ell_{t, \epsilon}^j (x) \mathrm{d} x$, the true intersection measure is singular once $d \ge 2$ and even its existence is nontrivial \cite{GemanHorowitzRosen1984}. Hence when approximating $\ell_t^{\otimes p}$, we do so in topologies generated by test functions. That is, we prove exponential approximations of the form
\[  
\limsup_{\epsilon \to 0} \limsup_{t \to \infty} \frac{1}{t} \log \mathbb{E} \exp |\langle f, \ell_{t}^{\otimes p} -\ell_{t, \epsilon}^{\otimes p} \rangle|^{1/p} = 0,
\]
where $f$ lies in some class of functions. If we let $\alpha_{\epsilon}([0, t]^p) := \ell_{t, \epsilon}^{\otimes p}(\mathbb{R}^d)$, equation \eqref{eq:mutual-exp-approx} corresponds to the case where $f$ is constant. Other classes previously considered include bounded nonnegative functions \cite{KoenigMoerters2006} and continuous compactly supported functions \cite{KoenigMukherjee2013, Mori2023}. We present a new proof technique that works for any bounded measurable function, generalizing all previous results. In spirit, our strategy is closest to Le Gall's approximation technique \cite{LeGall1986} in which one uses estimates of the Gaussian heat kernel to bound the moments of the integral. However, there are several complications coming from the mixed signs and singularities of the integrals which we overcome via original methods.
\begin{proposition}
    \label{prop:mutual-exp-approx}
    Suppose $p \ge 2$ and $d(p-1) < 2p$. For any bounded measurable function $f$ on $\mathbb{R}^d$,
    \[  
    \limsup_{\epsilon \to 0} \limsup_{t \to \infty} \frac{1}{t} \log \mathbb{E} \exp |\langle f, \ell_{t}^{\otimes p} -\ell_{t, \epsilon}^{\otimes p} \rangle|^{1/p} = 0.
    \]
\end{proposition}

Propositions~\ref{prop:self-exp-approx} and \ref{prop:mutual-exp-approx} are proven in Section~\ref{sec:exp-approx}. Using this approximation, we also obtain the following LDP for $t^{-p}\widetilde{\ell}_t^{\otimes p}$. Since $\ell_t^{\otimes p}$ may have arbitrary total mass, we extend the space $\widetilde{\mathcal{X}}_{\le 1}$ of \eqref{eq:MV-bdd-set} to the space $\widetilde{\mathcal{X}}$ of all finite measures (defined in \eqref{eq:MV-set}).

\begin{proposition}
    \label{prop:intersection-measure-LDP}
    Suppose $p \ge 2$ and $d(p-1) < 2p$. The distributions $t^{-p} \widetilde{\ell}_t^{\otimes p}$ satisfies an LDP in $(\widetilde{\mathcal{X}}(\mathbb{R}^d), \mathbf{D})$ defined in \eqref{eq:MV-set}, \eqref{eq:MV-metric} with good rate function
    \begin{equation}
        \label{eq:intersection-rate-function}
        \mathcal{I}^{\ell} (\zeta) = \begin{cases}
            \mathcal{I}(\xi^{\otimes p}) & \text{if } \mathcal{I}(\xi^{\otimes p}) < \infty \text{ and }\prod\limits_{j=1}^p \psi_i^j = \sqrt{\frac{d\gamma_i}{d x}} \\
            \infty & \text{otherwise},
        \end{cases}
    \end{equation}
    where $\zeta = \{ \widetilde{\gamma}_i\}_{i \in I} \in \widetilde{\mathcal{X}}(\mathbb{R}^d)$.
\end{proposition}

This is a generalization of \cite{Mukherjee2017}, which proved the case $(d, p) = (3, 2)$ modulo some minor differences in the topology $\widetilde{\mathcal{X}}$. A similar statement for all $(d, p)$ where $d(p-1) < 2p$ was done in \cite{Mori2023} for the vague topology.

We remark that the exponential approximation of Proposition~\ref{prop:mutual-exp-approx} is done in a topology even finer than $\widetilde{\mathcal{X}}$. However, this does not immediately give a stronger LDP for $t^{-p} \widetilde{\ell}_t^{\otimes p}$, since we do not have an LDP for the approximated measures $\widetilde{\ell}_{t, \epsilon}^{\otimes p}$.

Furthermore, via a similar process as in Theorems~\ref{thm:self-cond-ldp}--\ref{thm:mutual-tilted-ldp}, we can also derive an LDP for $t^{-p}\widetilde{\ell}_{t, \epsilon}^{\otimes p}$ conditional on $\ell_t^{\otimes p}(\mathbb{R}^d) \ge t^p$ or on the Gibbs measure tilted by (say) $\ell_t^{\otimes p}(\mathbb{R}^d)^{1/p}$. This process is rather routine, following similar arguments used to show Theorems~\ref{thm:self-cond-ldp} and \ref{thm:self-tilted-ldp}. As such, we have chosen not to present the details here.

\subsection{Outline and notation}

Our paper consists of four main steps:
\begin{enumerate}
    \item In Section~\ref{sec:exp-approx}, we prove that the approximations are exponentially good, in the sense of Propositions~\ref{prop:self-exp-approx} and \ref{prop:mutual-exp-approx}.
    \item In Section~\ref{sec:MV-top}, we establish an LDP for Brownian occupation measures on $\widetilde{\mathcal{X}}_{\le 1}^{\otimes p}$ and show that the approximations $\ell_{t, \epsilon}$ and $\ell_{t, \epsilon}^{\otimes p}$ are continuous functionals of the occupation measures.
    \item In Section~\ref{sec:main-proof}, we using the exponential approximation and the LDP of the previous sections to prove Theorems~\ref{thm:self-cond-ldp}--\ref{thm:mutual-tilted-ldp}, as well as Proposition~\ref{prop:intersection-measure-LDP}.
    \item Lastly, we prove Theorems~\ref{thm:self-cond-conv}--\ref{thm:mutual-tilted-conv} by characterizing the solutions to the variational problems arising from the LDPs. This step is deferred to the appendix.
\end{enumerate}

We conclude the introduction with some comments on our notation. Whenever we choose some $\xi^{\otimes p} \in \widetilde{\mathcal{X}}^{\otimes p}$, we automatically assume $\xi^{\otimes p} = \{ \widetilde{\alpha}_i^{\otimes p} \}_{i \in I}$ and often choose arbitrary representatives $\alpha_i^{\otimes p}$ for each $\widetilde{\alpha}_i^{\otimes p}$. Moreover, we take $(\psi_i^j)^2$ to be the densities of $\alpha_i^j$ (in cases where they exist). Under such conditions, we extend definitions on $\alpha_i^{\otimes p}$ or $\psi_i^j$ to $\xi^{\otimes p}$ in the ``natural'' way---some examples are the following.
\begin{itemize}
    \item $\| \xi^{\otimes p} \|_q^q = \sum_{i \in I} \| \alpha_i^{\otimes p} \|_q^q = \sum_{i \in I} \sum_{j=1}^p \| \psi_i^j \|_{2q}^{2q}$.
    \item $\xi^{\otimes p} \ast p_{\epsilon}^{\otimes p} = \{ \widetilde{\alpha}_i^{\otimes p} \ast p_{\epsilon}^{\otimes p} \}_{i \in I}$, where $\widetilde{\alpha}_i^{\otimes p} \ast p_{\epsilon}^{\otimes p}$ is the equivalence class of $(\alpha_i^1 \ast p_{\epsilon} , \dots , \alpha_i^p \ast p_{\epsilon} )$.
\end{itemize}

Such definitions will always be well-defined in the sense that they don't depend on the choice of representatives $\alpha_i^{\otimes p}$. Some of these values don't exist when $\mathcal{I}(\xi^{\otimes p}) = \infty$, but this is not of much concern (see Section~\ref{sec:main-proof}).

Given the new notation introduced in Section~\ref{subsec:intro-exp-approx}, the intersection local times may be written as
\[  
\beta([0, t]^q)^{1/q} = \| \ell_t \|_q = t \|L_t \|_q, \quad \alpha([0, t]^p) = \ell_t^{\otimes p}(\mathbb{R}^d) = \langle \mathbf{1}, \ell_t^{\otimes p} \rangle 
\]
The remainder of this paper will mostly be using the latter representations. We always work in the regime where $d(p-1) < 2p$. As we have already seen, $\delta$ denotes the Dirac delta at zero. We also use $\delta_x$ to denote the Dirac delta at $x$, mostly to write $\alpha \ast \delta_x$ as the translation of some measure or function. We also use the shorthand $\Delta_{\epsilon} = \delta -p_{\epsilon}$. $C_0(\mathbb{R}^d)$ denotes the space of continuous functions on $\mathbb{R}^d$ that vanish at infinity. We often take integrals on the ordered simplex
\[  
[0, t]_<^{m} = \{ (s^1, s^2, \dots ,s^m) : 0 < s^1 < s^2 <\dots < s^m < t \}.
\]
We use $C$ as a universal constant that may change from line to line. Unless otherwise stated, ``universal'' should be taken to mean that $C$ may depend on $d$ and $p$ (or $q$), but not anything else.

\subsection*{Acknowledgements}
We thank Amir Dembo for many helpful discussions, comments, and suggestions. We thank Chiranjib Mukherjee for his explanation of \cite{MukherjeeVaradhan2016}, which helped shape Section~\ref{sec:MV-top} of this paper. We thank Arka Adhikari, Izumi Okada, and Xia Chen for fruitful discussions. This work was supported by a grant from the Simons Foundation International [SFI-MPS-SDF-00014916]. Research partly funded by NSF grant DMS-2348142. 

    \section{Exponential approximation}
    \label{sec:exp-approx}

    In this section, we prove Propositions~\ref{prop:self-exp-approx} and \ref{prop:mutual-exp-approx}. We begin with a toy example where we approximate $\ell_t^{\otimes 2}(\mathbb{R})$. While not strictly necessary, this example illustrates our main ideas and also motivates the additional technical work needed to handle the general case.

    \subsection{Approximating $\ell_t^{\otimes 2}(\mathbb{R})$}

    We prove Proposition~\ref{prop:intersection-measure-LDP} in the special case where $f = 1$ and $(d, p) = (1, 2)$. That is, we show that
    \[
    \lim_{\epsilon \to 0} \limsup_{t \to \infty} \frac{1}{t} \log \mathbb{E} \exp |\ell_t^{\otimes 2}(\mathbb{R}) - \ell_{t, \epsilon}^{\otimes 2}(\mathbb{R})|^{1/2}  = 0.
    \]
     Our goal is to bound the moments of $\ell_t^{\otimes 2}(\mathbb{R}) - \ell_{t, \epsilon}^{\otimes 2}(\mathbb{R})$. To this end, observe that
    \begin{align*}
    \ell_t^{\otimes 2}(\mathbb{R}) - \ell_{t, \epsilon}^{\otimes 2}(\mathbb{R}) &= \int_{-\infty}^{\infty} \int_0^t \int_0^t \delta(W_s - x) \delta(\widetilde{W}_r - x) \mathrm{d} s \mathrm{d} r \mathrm{d} x - \int_{-\infty}^{\infty} \int_0^t \int_0^t p_{\epsilon}(W_s - x) p_{\epsilon}(\widetilde{W}_r - x) \mathrm{d} s \mathrm{d} r \mathrm{d} x \\
    &= \int_0^t \int_0^t \delta(W_s - \widetilde{W}_r) \mathrm{d} s \mathrm{d} r - \int_0^t \int_0^t p_{2 \epsilon}(W_s - \widetilde{W}_r) \mathrm{d} s dr \\
        &= \int_0^t \int_0^t \Delta_{2 \epsilon} (W_s - \widetilde{W}_r) \mathrm{d} s \mathrm{d} r,
    \end{align*}  
    where we use the shorthand $\Delta_{\epsilon} = \delta - p_{\epsilon}$ and $W, \widetilde{W}$ are independent Brownian motions. Thus, the $m$-th moment may be written as 
    \begin{align*}  
    \mathbb{E} \left( \int_0^t \int_0^t \Delta_{2 \epsilon} (W_s - \widetilde{W}_r) \mathrm{d} s \mathrm{d} r \right)^m &= \mathbb{E} \left[ \int_{[0, t]^{2m}} \prod_{i=1}^m \Delta_{2 \epsilon} (W(s_i) - \widetilde{W}(r_i) \mathrm{d} \mathbf{s} \mathrm{d} \mathbf{r} \right] \\
    &= \int_{[0, t]^{2m}} \mathbb{E} \left[ \prod_{i=1}^m \Delta_{2 \epsilon} (W(s_i) - \widetilde{W}(r_i) \mathrm{d} \mathbf{s} \mathrm{d}\mathbf{r} \right],
    \end{align*}
    which we bound in Lemma~\ref{lem:sketch-moment} below. Our main tools are the following Gaussian estimates. These are standard lemmas, but we include the proof for completeness.
    
    \begin{lemma}
        \label{lem:gaussian-estimates}
        Let $W_t$ be a Brownian motion in $\mathbb{R}^d$. For any $k < d$ and $x \in \mathbb{R}^d$, we have the following estimates. Here, $C$ is a constant that may depend on $d$ and $k$ but not on $t$, $x$, or $\epsilon$.
        \begin{align}
            \label{eq:delta-ineq}
            \mathbb{E} \delta(W_t - x) &\le C\min\{ |t|^{-d/2}, |x|^{-d} \} \\
            \label{eq:approx-ineq}
            \left| \mathbb{E} \Delta_{\epsilon}(W_t - x)\right| &\le C  \min\{ |t|^{-d/2}, |x|^{-d}, \epsilon|t|^{-d/2 - 1}, \epsilon|x|^{-d - 2} \} \\
            \label{eq:riesz-ineq}
            \mathbb{E}|W_t - x|^{-k} &\le C \min\{ t^{-k/2}, |x|^{-k} \} 
        \end{align}
    \end{lemma}

\begin{proof}
    The inequality \eqref{eq:delta-ineq} follows from the equation
    \[
    \mathbb{E} \delta(W_t - x) = p_t(x) = (2\pi t)^{-d/2} e^{-|x|^2 / 2t}  = |x|^{-d} \left\{ \pi^{-d/2} \Big( \frac{|x|^2}{2t} \Big)^{-d/2} e^{-|x|^2 / 2t}\right\}.
    \]
    The first equality show that $p_t(x) \le C |t|^{-d/2}$, while the second equality show that $p_t(x) \le C |x|^{-d}$ since the term $\bigl(\frac{|x|^2}{2t} \bigr)^{-d/2} \exp\bigl( - \frac{|x|^2}{2t} \bigr)$ is bounded.

    The inequality \eqref{eq:approx-ineq} follows similarly from the equation $\left| \mathbb{E} \Delta_{\epsilon}(W_t - x)\right| = |p_t(x) - p_{t + \epsilon}(x)|$. Indeed, the first two inequalities are immediate from the \eqref{eq:delta-ineq} and the triangle inequality. The second two come from 
    \[
        |p_t(x) - p_{t + \epsilon}(x)| \le \int_t^{t + \epsilon} \left| \frac{d}{\mathrm{d} s} p_s(x) \right| \mathrm{d} s = \int_t^{t + \epsilon} \left| (2\pi)^{-d/2} s^{-d/2-1} \left( - \frac{d}{2} + \frac{|x|^2}{2s} \right)e^{-|x|^2/2s}\right| \mathrm{d} s.
    \]
    Lastly, we show \eqref{eq:riesz-ineq}. The first inequality comes from
    \begin{align*}
        \mathbb{E}|W_t - x|^{-k} &= \int_{|y - x| \le \sqrt{t}} \frac{p_t(y)}{|y - x|^k} \mathrm{d} y + \int_{|y - x| > \sqrt{t}} \frac{p_t(y)}{|y - x|^{-k}} \mathrm{d} y \\
        &\le \int_{|y - x| \le \sqrt{t}} \frac{Ct^{-d/2} }{|y-x|^{k}}\mathrm{d} y + \int_{|y - x| > \sqrt{t}} \frac{p_t(y)}{(\sqrt{t})^{-k}} \mathrm{d} y \\
        &\le Ct^{-d/2} (\sqrt{t})^{-k + d} + t^{-k/2} \\
        &= C t^{-k/2}.
    \end{align*}
    The second is similar.
    \begin{align*}
        \mathbb{E}|W_t - x|^{-k} &= \int_{|y - x| \le |x|/2} \frac{p_t(y)}{|y - x|^k} \mathrm{d} y + \int_{|y - x| > |x| /2} \frac{p_t(y)}{|y - x|^{-k}} \mathrm{d} y \\
        &\le \int_{|y - x| \le |x| /2} \frac{p_t(x/2)}{|y-x|^{k}}\mathrm{d} y + \int_{|y - x| > |x| /2} C\frac{p_t(y)}{|x|^{-k}} \mathrm{d} y \\
        &\le C|x|^{-d} |x|^{-k + d} + |x|^{-k} \\
        &= C |x|^{-k}.
    \end{align*}
\end{proof}
    
    \begin{lemma}
    For any integer $m \ge 0$,
    \label{lem:sketch-moment}
        \begin{equation}
            \int_{[0, t]^{2m}} \left| \mathbb{E} \left[ \prod_{i=1}^m \Delta_{\epsilon}(W(s_i) - \widetilde{W}(r_i)) \right] \right|\mathrm{d} \mathbf{s} \mathrm{d}\mathbf{r} \le C^m \epsilon^{m/12} (m!)^2 \left( \frac{t^m}{m!} \right)^{17/12}
        \end{equation}
    \end{lemma}

    \begin{proof}
       Without loss of generality assume $s_1 \le s_2 \le \dots \le s_m$ and choose $\sigma \in S_m$ such that $r_{\sigma(1)} \le r_{\sigma(2)} \le \dots \le r_{\sigma(m)}$. We divide each interval $[s_i, s_{i+1}]$ into thirds and denote the times as $s_{i \pm 1/3} := s_i + (s_{i\pm 1} - s_i)/3$. We condition on the event $\widetilde{W}[0, t] \cup \{W(s_{i \pm 1/3})\}_{i=1}^m$. By the Markov property of Brownian motions, each term $W(s_i)$ becomes an independent variable distributed as a point on the Brownian bridge from $W(s_{i -1/3})$ to $W(s_{i+1/3})$. Therefore, $W(s_i)$ is Gaussian with mean
        \[
        \overline{W}(s_i) := W(s_{i - 1/3}) + \frac{s_i - s_{i - 1/3}}{s_{i + 1/3} - s_{i - 1/3}} (W(s_{i+1/3}) - W(s_{
        i - 1/3}))
        \]
        and variance
        \[
        \frac{(s_{i + 1/3} - s_i)(s_i - s_{i - 1/3})}{s_{i + 1/3} - s_{i - 1/3}} = \left( \frac{1}{s_{i+1/3} - s_i} + \frac{1}{s_i - s_{i-1/3}} \right)^{-1} = \frac{1}{3}\left( \frac{1}{s_{i+1} - s_i} + \frac{1}{s_i - s_{i-1}} \right)^{-1}.
        \]
        By conditioning on $\widetilde{W}[0, t] \cup \{W(s_{i \pm 1/3})\}_{i=1}^m$ and applying Lemma \ref{lem:gaussian-estimates}, we have
        \begin{equation}
        \label{eq:sketch-approx}
        \begin{aligned}
            \biggl| \mathbb{E} \biggl[ \prod_{i=1}^m \Delta_{\epsilon}(W(s_i) &- \widetilde{W}(r_i)) \biggr] \biggr| \\
            &= \biggl|\mathbb{E} \biggl[ \prod_{i=1}^m  \mathbb{E} \Bigl[\Delta_{\epsilon}(W(s_i) - \widetilde{W}(r_i)) \Big| W(s_{i \pm 1/3}), \widetilde{W}(r_i) \Bigr]  \biggr]\biggr| \\
            &\le \mathbb{E} \biggr[ \prod_{i=1}^m  \biggl| \mathbb{E} \Bigl[\Delta_{\epsilon}(W(s_i) - \widetilde{W}(r_i)) | W(s_{i \pm 1/3}), \widetilde{W}(r_i) \Bigr] \biggr| \biggr] \\
            &\le C^m \epsilon^{m/12} \prod_{i=1}^m \bigl( |s_i - s_{i-1}|^{-1/4} + |s_{i+1} - s_i|^{-1/4} \bigr) \mathbb{E} \biggl[ \prod_{i=1}^m |\overline{W}(s_i) - \widetilde{W}(r_i)|^{-2/3}\biggr]
        \end{aligned}
        \end{equation}
        The third line uses the conditional distribution of $W(s_i)$ and inequality
        \[
        |\mathbb{E} \Delta_{\epsilon}(W_t - x)| \le C \epsilon^{-1/12} |t|^{-1/4} |x|^{-2/3}
        \]
        which comes from H\"{o}lder's inequality applied to equation~\eqref{eq:approx-ineq} with $d = 1$ and weights $\frac{1}{4}, \frac{2}{3}, \frac{1}{12}, 0$ respectively.

        Note that the term inside the expectation is now nonnegative. At this point, we condition iteratively on all but the last point in $\widetilde{W}$. That is, by conditioning on $W[0, t] \cup \widetilde{W}[0, r_{\sigma(m-1})]$, $\widetilde{W}(r_{\sigma(m)})$ becomes a Gaussian with mean $\widetilde{W}(r_{\sigma(m-1)})$ and variance $r_{\sigma(m)} - r_{\sigma(m-1)}$. Therefore, we can apply \eqref{eq:delta-ineq} iteratively to get
        \begin{equation}
        \label{eq:sketch-riesz}
        \begin{aligned}  
        \mathbb{E} \bigg[ \prod_{i=1}^m &|\overline{W}(s_i) - \widetilde{W}(r_i) |^{-2/3} \bigg]  \\&= \mathbb{E} \left[  \prod_{i=1}^{m-1} |\overline{W}(s_{\sigma(i)}) - \widetilde{W}(r_{\sigma(i)})|^{-2/3} \mathbb{E} \left[ |\overline{W}(s_{\sigma(m)}) - \widetilde{W}(r_{\sigma(m)})|^{-2/3} \Big| \overline{W}(s_{\sigma(m)}), \widetilde{W}(r_{\sigma(m-1)}) \right]\right] \\
        &\le C|r_{\sigma(m)} - r_{\sigma(m-1)}|^{-1/3} \mathbb{E} \left[ \prod_{i=1}^{m-1}|\overline{W}(s_{\sigma(i)}) - \widetilde{W}(r_{\sigma(i)})|^{-2/3} \right] \\
        &\vdots \\
        &\le C^m \prod_{i=1}^m |r_{\sigma(m)} - r_{\sigma(m-1)}|^{-1/3}.
        \end{aligned}
        \end{equation}
        Combined with \eqref{eq:sketch-approx}, we obtain
        \[  
        \bigg| \mathbb{E} \Big[ \prod_{i=1}^m \Delta_{\epsilon}\big( W(s_i) - \widetilde{W}(r_i) \big) \Big] \bigg| \le C^m \prod_{i=1}^m |r_{\sigma(m)} - r_{\sigma(m-1)}|^{-1/3} \prod_{i=1}^m \Big( |s_i -s_{i-1}|^{-1/4} + |s_{i+1} - s_i|^{-1/4} \Big).
        \]
        We can integrate both sides on the simplex
        \[  
        ([0, t]_<^m)^2 = \{ (s_1 ,\dots , s_m , r_1 , \dots , r_m) : s_1 < \dots < s_m , r_{\sigma(1)} < \dots < r_{\sigma(m)} \}
        \]
        using the Dirichlet integral (Lemma~\ref{lem:dirichlet-integral}) to get
        \[  
        \int_{([0, t]_<^m)^2} \bigg| \mathbb{E} \Big[ \prod_{i=1}^m \Delta_{\epsilon}(W(s_i) - \widetilde{W}(r_i)) \Big] \bigg| \mathrm{d} \mathbf{s} \mathrm{d} \mathbf{r} \le C^m \epsilon^{m/12} \frac{t^{2m/3}}{\Gamma(\frac{2}{3}m + 1)}  \times \frac{t^{3m/4}}{\Gamma(\frac{3}{4}m + 1)} \le C^m \epsilon^{m/12} \left( \frac{t^m}{m!} \right)^{17/12}.
        \]
        Note that the product $\prod_{i=1}^m \left( |s_i -s_{i-1}|^{-1/3} + |s_{i+1} - s_i|^{-1/3} \right)$ only has $2^m$ terms, so the combinatorial factor gets absorbed in the $C^m$ term. We complete the proof by multiplying $(m!)^2$ to account for all possible orderings of $\{s_i\}$ and $\{r_i\}$.
    \end{proof}

    \begin{lemma}[Dirichlet integral {\cite[Chapter~12.5]{WhittakerWatson2021}}]
        For any $\alpha_1, \dots , \alpha_m > 0$,
        \label{lem:dirichlet-integral}
        \begin{equation}
        \label{eq:dirichlet-integral}
    \int_{[0, t]_<^m} \prod_{i=1}^m (s_i - s_{i-1})^{\alpha_i - 1} d\mathbf{s} = \frac{\Gamma(\alpha_1) \Gamma(\alpha_2)\dots \Gamma(\alpha_m)}{\Gamma(\alpha_1 + \alpha_2 + \dots + \alpha_m + 1)} t^{\alpha_1 + \alpha_2 + \dots + \alpha_m}.
    \end{equation}
    \end{lemma}
    
    \begin{corollary}
    \label{cor:sketch-exp-approx}
        \[
        \lim_{\epsilon \to 0} \limsup_{t \to \infty} \frac{1}{t} \log \mathbb{E} \exp | \ell_{t}^{\otimes 2}(\mathbb{R}) - \ell_{t, \epsilon}^{\otimes 2}(\mathbb{R})|^{1/2} = 0.
        \]        
    \end{corollary}

    \begin{proof}
        We know from Lemma~\ref{lem:sketch-moment} that
        \[  
        \mathbb{E} \left| \ell_{t}^{\otimes 2}(\mathbb{R}) - \ell_{t, \epsilon}^{\otimes 2}(\mathbb{R})\right|^m \le C^m (m!)^2 \left( \frac{\epsilon^{m/17} t^m}{m!} \right)^{17/12}.
        \]
        Note that we have moved the absolute value inside the expectation. When $m$ is even, this is obviously valid. When $m$ is odd, we can use the bound $((m+1)!)^{m/(m+1)} \le C^m m!$ along with H\"{o}lder's inequality on the $m+1$ case (we may similarly generalize to fractional moments).
        Therefore,
        \begin{align*}
            \mathbb{E} \exp | \ell_{t}^{\otimes 2}(\mathbb{R}) - \ell_{t, \epsilon}^{\otimes 2}(\mathbb{R})|^{1/2} &= \sum_{m=0}^\infty \frac{\mathbb{E} | \ell_{t}^{\otimes 2}(\mathbb{R}) - \ell_{t, \epsilon}^{\otimes 2}(\mathbb{R})|^{m/2}}{m!}  \\
            &\le \sum_{m=0}^\infty \frac{1}{m!} \left( \mathbb{E} | \ell_{t}^{\otimes 2}(\mathbb{R}) - \ell_{t, \epsilon}^{\otimes 2}(\mathbb{R})|^m \right)^{1/2} \\
            &\le \sum_{m=0}^\infty \left( \frac{(C \epsilon^{1/17} t)^m}{m!} \right)^{17/24} \\
            &= \exp \{ O(\epsilon^{1/17} t)\}.
        \end{align*}
        The last line comes from e.g., \cite[Section~8.8]{Olver1997}.
    \end{proof}

    In short, our main idea is as follows. By conditioning on small intervals around each $W(r_i)$, we may use the independence of each increment to move the absolute value inside the expectation. From there, we iteratively exchange the randomness of $|W_t - x|^{-k}$ for a deterministic factor of $|t|^{-k/2}$. From this perspective, $\delta$ behaves similarly to $| \cdot |^{-d}$ and $\Delta_{\epsilon}$ to $| \cdot |^{-d-2}$. After all exchanges have been made, we integrate both sides using Dirichlet's integral. The key point here is that the bounds are indeed integrable, i.e., that we never get terms of order $|t|^{-1}$ or more. We remark that for this proof, we never used the randomness coming from the middle thirds of the intervals $[s_i, s_{i+1}]$.

    This philosophy continues to apply in the general case. The same conditioning argument lets us bound $\mathbb{E} \langle f, \ell_t^{\otimes p} - \ell_{t, \epsilon}^{\otimes p} \rangle^m$ by an expectation over nonnegative products; the sign of $f$ does not cause any issues. The more pressing problem is that in higher dimensions, the factors become more singular. This means we have to be more careful when bounding the expectations. To this end, we explain two ways in which the above proof is wasteful and how we refine them.
    
    The first is the conditioning over the endpoints $W(s_{i \pm 1/3})$. Because $\overline{W}(s_i)$ is distributed as a Brownian bridge, its variance is of order $\min \{s_i - s_{i-1}, s_{i+1} - s_i \}$. In other words, while the \emph{average} order is $t$, we have to account for terms of order $t^2$ since small intervals get counted twice. To avoid this waste, we should exchange as little as possible in the inequality \eqref{eq:sketch-approx}, and instead leave more singularity in the product inside the expectation (which we can do by altering the weights used when applying H\"{o}lder's inequality to \eqref{eq:approx-ineq}).

    But this is only postponing the issue, as the second challenge concerns the terms $|\overline{W}(s_i) - \widetilde{W}(r_i)|^{-k}$. In \eqref{eq:sketch-riesz}, we used the randomness of $\widetilde{W}$ to exchange its expectation for $|r_i-r_{i-1}|^{-2/3}$. In higher dimensions when the orders are more singular, this strategy no longer gives us an integrable bound. The solution is to use the randomness coming from \emph{both} $W$ and $\widetilde{W}$. In this way, we can split the expectation $\mathbb{E}|\overline{W}(s_i) - \widetilde{W}(r_i)|^{-k}$ into (say) $|s_i - s_{i-1}|^{-k/4} |r_i - r_{i-1}|^{-k/4}$. This is why we split each interval $[s_i , s_{i+1}]$ into thirds---even after conditioning on Brownian bridges, the randomness from the middle third remains untouched. This means that it remains available for us to use at this later stage. Because the integrand is nonnegative, we are able to condition iteratively on the last point instead of both endpoints, as we had to do with Brownian bridges.

    \begin{remark*}
        This proof method fundamentally breaks down once we reach criticality at $d(p-1) = 2p$. The $(p-1)$ Dirac delta functions introduce singularities of order $|t|^{-d/2}$ each, culminating in a singularity of order $|t|^{-d(p-1)/2}$. If we split evenly among the $p$ time variables, we get a singularity of order $|t|^{-d(p-1)/2p}$ for each variable. Thus the integral is only finite when $d(p-1)/2p < 1$, or equivalently, $d(p-1) < 2p$.
    \end{remark*}
    
    \subsection{Approximating $\ell_t$}

    Now we explain the self-intersection case and prove Proposition~\ref{prop:self-exp-approx}. For reasons identical to Corollary~\ref{cor:sketch-exp-approx}, it suffices to prove the moment bounds of Corollary~\ref{cor:self-approx-momemt} below. We first prove a moment estimate for integer values $p$, and then use as interpolation argument to generalize to all $q > 1$. Since $\ell_t(x) - \ell_{t, \epsilon}(x) = \int_0^t \Delta_{\epsilon}(W_s - x) \mathrm{d} s$, we may write
    \begin{align*}
        \left(\int_{-\infty}^{\infty} (\ell_t(x) - \ell_{t, \epsilon}(x))^p \mathrm{d} x \right)^{1/p} &= \left[ \int_{- \infty}^{\infty} \left(\int_0^t \Delta_{\epsilon}(W_s - x) dt \right)^p \mathrm{d} x\right]^{1/p} \\
        &= \bigg[ \int_{-\infty}^{\infty} \int_{[0, t]^p} \prod_{j=1}^p \Delta_{\epsilon}(W(s^j) - x) \mathrm{d} \mathbf{s} \mathrm{d} x\bigg]^{1/p} \\
        &= \bigg[ p! \times  \int_{-\infty}^{\infty} \int_{[0, t]_<^p} \prod_{j=1}^p \Delta_{\epsilon}(W(s^j) - x) \mathrm{d} \mathbf{s} \mathrm{d} x\bigg]^{1/p},
    \end{align*}
    where the simplex $[0, t]_<^p$ denotes
    \[   
    [0, t]^p_< := \{ (s^1 , s^2, \dots , s^p): s^1 < s^2 < \dots < s^p \}.
    \]
    Thus, it suffices to show the following lemma.
    \begin{lemma}
    \label{lem:self-approx-moment}
    For any integer $p \ge 2$ and $0 < \theta < 1/6$,
        \begin{equation}
            \label{eq:self-moment}
        \Bigg| \mathbb{E} \bigg[ \int_{-\infty}^{\infty}\int_{[0, t]_<^{p}} \prod_{j=1}^p\Delta_{\epsilon}(W(s^{j}) - x)  \mathrm{d} \mathbf{s} \mathrm{d} x
        \bigg]^m \Bigg| \le C^m \epsilon^{\theta(p-1) m} (m!)^p \Big( \frac{t^m}{m!} \Big)^{\frac{p+1}{2} - (p-1)\theta}.
        \end{equation}
    \end{lemma}

    \begin{proof}
        Note that
        \begin{multline*}
        \int_{-\infty}^{\infty}\int_{[0, t]_<^{p}} \prod_{j=1}^p\Delta_{\epsilon}(W(s^{j}) - x)  \mathrm{d} \mathbf{s} \mathrm{d} x \\
        = \int_{-\infty}^{\infty}\int_{[0, t]_<^{p}}\delta(W_t(s^1) - x) \prod_{j=2}^p\Delta_{\epsilon}(W(s^{j}) - x)  \mathrm{d} \mathbf{s} \mathrm{d} x - \int_{-\infty}^{\infty}\int_{[0, t]_<^{p}}  p_{\epsilon}(W_t(s^1) - x)\prod_{j=2}^p\Delta_{\epsilon}(W(s^{j}) - x)  \mathrm{d}\mathbf{s} \mathrm{d} x.
        \end{multline*}
        It suffices to show the moment bounds for each term separately, i.e.,
        \begin{equation}
        \label{eq:self-approx-1}
            \Biggl| \mathbb{E} \biggl[ \int_{-\infty}^{\infty}\int_{[0, t]_<^p} \delta(W(s^1) - x) \prod_{j=2}^p \Delta_{\epsilon}(W(s^j) -x) \mathrm{d} \mathbf{s} \mathrm{d} x \biggr]^m \Biggr| \le C^m \epsilon^{\theta(p-1) m} (m!)^p \Bigl( \frac{t^m}{m!} \Bigr)^{\frac{p+1}{2} - (p-1)\theta},
        \end{equation}
        \begin{equation}
        \label{eq:self-approx-2}
            \Biggl| \mathbb{E} \biggl[ \int_{-\infty}^{\infty}\int_{[0, t]_<^p} p_{\epsilon}(W(s^1) - x) \prod_{j=2}^p \Delta_{\epsilon}(W(s^j) -x) \mathrm{d} \mathbf{s} \mathrm{d} x \biggr]^m \Biggr| \le C^m \epsilon^{\theta(p-1) m} (m!)^p \Bigl( \frac{t^m}{m!} \Bigr)^{\frac{p+1}{2} - (p-1)\theta}.
        \end{equation}

        The proof for each are similar so we only explain \eqref{eq:self-approx-1} in detail, with the modifications for proving \eqref{eq:self-approx-2} mentioned in the last paragraph. By integrating over $x$, we may write the left hand side of \eqref{eq:self-approx-1} as
        \[
        \Biggl| \mathbb{E} \biggl[ \int_{([0, t]_<^p)^m} \prod_{i=1}^m  \prod_{j=2}^p \Delta_{\epsilon}(W(s^j_i) -W(s^1_i)) d \mathbf{s}  \biggr]^m \Biggr|
        \le \int_{([0, t]_<^p)^m} \biggl| \mathbb{E} \biggl[ \prod_{i=1}^m \prod_{j=2}^p \Delta_{\epsilon}(W(s^j_i) -W(s^1_i))\biggr]^m \biggr| d \mathbf{s}  .
        \]

        The only restrictions on the orders of $\{ s_i^j \}$ are $s_i^1 < s_i^2 < \dots < s_i^p$ for each $i$, and we cannot make additional assumptions without losing generality. As such, we need to consider all $(mp)! / (p!)^m$ possible orderings of times $\{s^1_1 , \dots , s^1_m , \dots , s^p_m \}$ separately. For each ordering, define $\tau(s^j_i)$ to be the time appearing immediately before $s^j_i$ in the set $\{s^1_1 , \dots , s^1_m , \dots , s^p_m \}$. Clearly, $\tau^{-1}$ would map $s^j_i$ to the time immediately after $s^j_i$. Divide each interval $[\tau(s^j_i), s^j_i]$ into thirds and label the timestamps
        \[
        s^j_{i - 1/3} := s^j_i + \frac{\tau(s^j_i) - s^j_i}{3}, \quad s^j_{i + 1/3} := s^j_i + \frac{\tau^{-1}(s^j_i) - s^j_i}{3}.
        \]
        Conditioned on the event $\{ W(s^1_i) : 1 \le i \le m \} \cup \{ W(s^j_{i \pm 1/3}) : 2 \le j \le p, 1 \le i \le m\}$, each $W(s_i^j)$ is distributed as a Gaussian with mean
        \[
        \overline{W}(s_i^j) := W(s_{i - 1/3}^j) + \frac{s_i^j - s_{i - 1/3}^j}{s_{i + 1/3}^j - s_{i - 1/3}^j} (W(s_{i+1/3}^j) - W(s_{
        i - 1/3}^j))
        \]
        and variance
        \[
        \frac{(s^j_{i + 1/3} - s^j_i)(s^j_i - s^j_{i - 1/3})}{s^j_{i + 1/3} - s^j_{i - 1/3}} = \left( \frac{1}{s^j_{i+1/3} - s^j_i} + \frac{1}{s^j_i - s^j_{i-1/3}} \right)^{-1} = \frac{1}{3}\left( \frac{1}{\tau^{-1}(s^j_i) - s^j_i} + \frac{1}{s^j_i - \tau(s^j_i)} \right)^{-1}.
        \]

        By conditioning on $\{ W(s^1_i) : 1 \le i \le m \} \cup \{ W(s^j_{i \pm 1/3}) : 2 \le j \le p, 1 \le i \le m\}$,
        \begin{equation}
        \label{eq:self-approx}
        \begin{aligned}
            &\Biggl| \mathbb{E} \biggl[ \prod_{i=1}^m   \prod_{j=2}^p\Delta_{\epsilon}(W(s^j_i) - W(s^1_i))\biggr] \Biggr| \\
            &= \Biggl| \mathbb{E} \biggl[ \prod_{i=1}^m \prod_{j=2}^p \mathbb{E} \Bigl[ \Delta_{\epsilon}\Bigl(W(s^j_i) - W(s^1_i)\Bigr) \Big| W(s^j_{i \pm 1/3}), W(s^1_i)\Bigr]\biggr] \Biggr| \\
            &\le \mathbb{E} \Biggl[ \prod_{i=1}^m \prod_{j=2}^p \biggl| \mathbb{E} \Bigl[ \Delta_{\epsilon}\Bigl(W(s^j_i) - W(s^1_i)\Bigr) \Big| W(s^j_{i \pm 1/3}), W(s^1_i)\Bigr] \biggr| \Biggr] \\
            &\le C^m \epsilon^{\theta (p-1)m} \prod_{i=1}^m \prod_{j=2}^p \Bigl( |s^j_i - \tau(s^j_i)|^{-1/6 -\theta} + |\tau^{-1}(s^j_i) - s^j_i|^{-1/6 - \theta} \Bigr) \mathbb{E} \biggl[ \prod_{i=1}^m \prod_{j=2}^p |\overline{W}(s^j_i) - W(s^1_i)|^{-2/3} \biggr].
        \end{aligned}
        \end{equation}
        The last line uses the inequality
        \[
        |\mathbb{E} \Delta_{\epsilon}(W_t - x)| \le C \epsilon^{\theta} |t|^{-1/6 - \theta} |x|^{-2/3},
        \]
        which comes from H\"{o}lder's inequality applied to equation~\eqref{eq:approx-ineq} with $d = 1$ and weights $ \frac{1}{3} - \theta, \frac{2}{3}, \theta, 0$ respectively.

        Now let $s^{j_1}_{i_1}$ be the largest time out of $\{s^j_i : 2 \le j \le p, 1 \le i \le m\}$ and $\tau(s^{j_1}_{i_1}) = s^{j_0}_{i_0}$. By conditioning on $W[0, s^{j_0}_{i_0 + 1/3}]$, all points except $\overline{W}(s^{j_1}_{i_1})$ are completely determined, and
        \[
        \overline{W}(s^{j_1}_{i_1}) = W(s^{j_0}_{i_0 + 1/3}) + \Big( W(s_{i_1 - 1/3}^{j_1}) - W(s_{i_0 + 1/3}^{j_0}) \Bigr) + \frac{s_{i_1}^{j_1} - s_{i_1 - 1/3}^{j_1}}{s_{i_1 + 1/3}^{j_1} - s_{i_1 - 1/3}^{j_1}} \Big( W(s_{i_1 + 1/3}^{j_1}) - W(s_{i_1 - 1/3}^{j_1}) \Bigr)
        \]
        has mean $W(s^{j_0}_{i_0 + 1/3})$ and variance greater than $s^{j_1}_{i_1 - 1/3} - s^{j_0}_{i_0 + 1/3} = (s^{j_1}_{i_1} - s^{j_0}_{i_0})/3$. Therefore,
        \begin{equation}
        \begin{aligned}
        \label{eq:self-approx-riesz}
         \mathbb{E} &\left[ \prod_{i=1}^m \prod_{j=2}^p |\overline{W}(s^j_i) - W(s^1_i)|^{-2/3} \right] \\
         &= \mathbb{E} \biggl[  \mathbb{E}\left[|\overline{W}(s^{j_1}_{i_1}) - W(s^1_{i_1})|^{-2/3} \big| W(s^1_{i_1}) , W(s^{j_0}_{i_0 + 1/3}) \right]\prod_{(i, j) \ne (i_1, j_1)} |\overline{W}(s^j_i) -W(s^1_i)|^{-2/3}  \biggr] \\
         &\le |s^{j_1}_{i_1} - \tau(s^{j_1}_{i_1})|^{-1/3} \mathbb{E}\biggl[ \prod_{(i, j) \ne (i_1 , j_1)} |\overline{W}(s^j_i) - W(s^1_i)|^{-2/3} \biggr].
        \end{aligned}
        \end{equation}
        The last line uses the inequality \eqref{eq:riesz-ineq}. Repeating this $m(p-1)$ times gives us
        \[  
        \mathbb{E} \biggl[ \prod_{i=1}^m \prod_{j=2}^p |\overline{W}(s^j_i) - W(s^1_i)|^{-2/3} \biggr] \le C^m  \prod_{i=1}^m \prod_{j=2}^p |s^j_i - \tau(s^j_i)|^{-1/3}.
        \]
        Note that all terms containing $W(s_i^1)$ go away as we take expectations over $W(s_i^p), \dots W(s_i^2)$, which appear before $s^1_i$ thanks to the ordering $s^1_i < s^2_i \dots < s^p_i$. Combined with \eqref{eq:self-approx}, this yields
        \begin{multline*}
        \Biggl| \mathbb{E} \biggl[ \prod_{i=1}^m \prod_{j=2}^p\Delta_{\epsilon}(W(s^j_i) - W(s^1_i))\biggr] \Biggr| \\
        \le C^m \epsilon^{\theta(p-1)m} \prod_{i=1}^m \prod_{j=2}^p |s^j_i -\tau(s^j_i)|^{-1/3} \bigl( |s^j_i - \tau(s^j_i)|^{-1/6 - \theta} + |\tau^{-1}(s^j_i) - s^j_i|^{-1/6 - \theta} \bigr).
        \end{multline*}
        The only remaining step is to integrate both sides. For a fixed ordering of $\{ s^j_i\}$, we may apply Dirichlet's integral to bound the right hand side by $C^m \epsilon^{\theta(p-1)m} ( \frac{t^m}{m!} )^{\frac{p+1}{2} - (p-1)\theta}$ as long as $\theta < 1/6$. Since there are $\frac{(mp)!}{(p!)^m} = O(C^m (m!)^p)$ possible orderings, we can conclude that
        \[
        \Biggl| \mathbb{E} \biggl[ \prod_{i=1}^m \prod_{j=2}^p\Delta_{\epsilon}\bigl(W(s^j_i) - W(s^1_i)\bigr)\biggr] \Biggr| \le C^m \epsilon^{\theta(p-1)m} (m!)^p \Bigl( \frac{t^m}{m!} \Bigr)^{\frac{p+1}{2} - (p-1)\theta}.
        \]

        For the proof of \eqref{eq:self-approx-2}, we simply note that $p_{\epsilon}(x) = \mathbb{E} \delta(\sqrt{\epsilon}Z - x)$, where $Z$ is a standard Gaussian. Therefore, we may write
        \begin{align*}
        \int p_{\epsilon} (W(s^1) - x) \prod_{j=2}^p \Delta_{\epsilon}(W(s^j) -x)  \mathrm{d} x &=  \mathbb{E} \biggl[ \int \delta(W(s^1) + \sqrt{\epsilon}Z- x) \prod_{j=2}^p \Delta_{\epsilon}(W(s^j) -x) \mathrm{d} x \biggr] \\
        &= \mathbb{E} \biggl[  \prod_{j=2}^p \Delta_{\epsilon}(W(s^j) - W(s^1) - \sqrt{\epsilon}Z) \biggr],
        \end{align*}
        where $Z$ is a Gaussian independent of $W$ and the expectation is taken over $Z$. Thus the $m$-th moment may be written as
        \[  
        \mathbb{E} \biggl[ \prod_{i=1}^m \prod_{j=2}^p\Delta_{\epsilon}\Bigl(W(s^j_i) - W(s^1_i) - \sqrt{\epsilon}Z_i\Bigr)\biggr],
        \]
        where $Z_1 , \dots , Z_m$ are independent standard Gaussian variables which are also independent from $W$. From this point, we can proceed exactly as before to obtain the same bound.
    \end{proof}

    \begin{corollary}
        \label{cor:self-approx-momemt}
        There exists sufficiently small $\theta > 0$ such that for any $m \ge 0$ and $q > 1$,
        \[  
        \mathbb{E} \| \ell_t - \ell_{t, \epsilon} \|_q^m \le C^m \epsilon^{ \theta m} m! \left(\frac{t^m}{m!} \right)^{\frac{q+1}{2} - (q-1) \theta}.
        \]
    \end{corollary}
    
    \begin{proof}
        When $q$ is an even integer, we have $|\ell_t - \ell_{t, \epsilon}|^q = (\ell_t - \ell_{t, \epsilon})^q$ so the the above is a direct consequence of Lemma~\ref{lem:self-approx-moment}. For general $q > 1$, we interpolate between $1$ and a large even number. Since $\| \ell_t \|_1 = \| \ell_{t, \epsilon }\|_1 = t$, we immediately have
        \[      
        \mathbb{E} \|\ell_t - \ell_{t, \epsilon} \|_1^m \le C^m t^m = C^m m! \left( \frac{t^m}{m!} \right).
        \]
        Let $p = 2 \lfloor q/2 \rfloor$ and $\eta \in (0, 1]$ such that $q = (1 - \eta) + \eta p$. By H\"{o}lder's inequality,
        \begin{align*}
        \mathbb{E} \|\ell_t - \ell_{t, \epsilon} \|_q^m &\le \mathbb{E} (\| \ell_t - \ell_{t, \epsilon} \|_1^{1 - \eta} \|\ell_t - \ell_{t, \epsilon} \|_p^{\eta})^m \\
        &\le \left( \mathbb{E}\| \ell_t - \ell_{t, \epsilon} \|_1^m \right)^{1 - \eta} \left( \mathbb{E}\| \ell_t - \ell_{t, \epsilon} \|_p^m \right)^{\eta} \\
        &\le C^m \epsilon^{\eta \theta m} m! \left( \frac{t^m}{m!} \right)^{\frac{q+1}{2} - (q-1) \eta \theta} \\
        \end{align*}
        so we are done.
    \end{proof}
    \subsection{Approximating intersection measures}

    We now turn to the intersection measures $\ell_t^{\otimes p}$. For the same reasons as in Corollary~\ref{cor:sketch-exp-approx}, it is enough to prove a sufficient moment bound on $\langle f, \ell_t^{\otimes p} - \ell_{t, \epsilon}^{\otimes p} \rangle$. We may write this as an interpolating sum as follows.
    \begin{align*}
    \langle f, \ell_t^{\otimes p} - \ell_{t, \epsilon}^{\otimes p} \rangle &= \int_{\mathbb{R}^d} \int_{[0, t]^p} f(x) \prod_{j=1}^p \delta(W^j(s^j) - x) \mathrm{d}\mathbf{s} \mathrm{d} x - \int_{\mathbb{R}^d} \int_{[0, t]^p} f(x) \prod_{j=1}^p p_{\epsilon} (W^j(s^j) - x) \mathrm{d}\mathbf{s} \mathrm{d} x \\
    &= \sum_{j_0 = 1}^p \int_{\mathbb{R}^d} \int_{[0, t]^p}f(x) \Delta_{\epsilon}(W^j(s^j) - x) \prod_{j <j_0} \delta(W^j(s^j) - x) \prod_{j >j_0} p_{\epsilon}(W^j (s^j) - x) \mathrm{d}\mathbf{s} \mathrm{d} x.
    \end{align*}
    Therefore, it suffices to bound the moments of each of the $p$ summands. The case $j_0 = p$ is stated as Lemma~\ref{lem:mutual-approx-moment}, and the rest can be done similarly (cf. the last paragraph of Lemma~\ref{lem:self-approx-moment}).
    
    We break up the proof into smaller lemmas. Lemma~\ref{lem:mutual-basic} gives us preliminary estimates, and Lemmas~\ref{lem:mutual-chain-1} and \ref{lem:mutual-chain-2} serve as the analog of \eqref{eq:sketch-riesz} when combined.

    \begin{lemma}
    \label{lem:mutual-basic}
        Let $W_t$ be a Brownian motion in $\mathbb{R}^d$. For any $0 \le \theta \le d$ and $x, y \in \mathbb{R}^d$,
        \begin{equation}
        \label{eq:mutual-basic-delta}
            \mathbb{E} \delta(W_t -x) |W_t - y|^{-\theta} \le C |t|^{-(d-\theta)/2} |x - y|^{-\theta} |x|^{-\theta}.
        \end{equation}
        Similarly, for any $0 < k < d$ and $0 \le \theta \le k/2$,
        \begin{align}
        \label{eq:mutual-basic-riesz-1}
        \mathbb{E} |W_t -x|^{-k} |W_t - y|^{-k} &\le C |t|^{-k/2} |x-y|^{-k}, \\
        \label{eq:mutual-basic-riesz-2}
            \mathbb{E} |W_t -x|^{-k} |W_t - y|^{-\theta} &\le C |t|^{-(k-\theta)/2} |x|^{-\theta}|x-y|^{-\theta}.
        \end{align}
        Here, $C$ may depend on $d$ and $k$ but not on $t$, $x$, $y$, or $\theta$.
    \end{lemma}

    \begin{proof}
        The proof for \eqref{eq:mutual-basic-delta} is straightforward from \eqref{eq:delta-ineq} since
        \[  
        \mathbb{E} \delta(W_t - x)|W_t - y|^{- \theta} = p_t(x) |x - y|^{-\theta} \le C |t|^{-(d-\theta)/2} |x|^{-\theta} |x-y|^{-\theta}.
        \]

        To show \eqref{eq:mutual-basic-riesz-1}, we use the triangle inequality $|x - y|^{k} \le C(|W_t - x|^{k} + |W_t - y|^k)$. Hence by \eqref{eq:riesz-ineq},
        \[  
        |x - y|^{k} \mathbb{E} |W_t - x|^{-k} |W_t - y|^{-k} \le C \left( \mathbb{E} |W_t - x|^{-k} + \mathbb{E} |W_t - y|^{-k} \right) \le C |t|^{-k/2}
        \]
        Lastly, for \eqref{eq:mutual-basic-riesz-2}, we use H\"{o}lder's inequality on \eqref{eq:riesz-ineq} and \eqref{eq:mutual-basic-riesz-1} to get
        \begin{align*}
        \mathbb{E} |W_t -x|^{-k} |W_t - y|^{-\theta} &\le \left( \mathbb{E} |W_t - x|^{-k}\right)^{(k - \theta)/k} \left(\mathbb{E} |W_t - x|^{-k} |W_t - y|^{-k}\right)^{\theta/k} \\
        &\le C |t|^{-\frac{(k - 2\theta)k}{2(k-\theta)}} |x|^{-\frac{\theta k}{k - \theta}})^{(k-\theta)/k}(|t|^{-k/2} |x-y|^{-k})^{\theta/k} \\
        &= C |t|^{-(k - \theta)/2} |x|^{-\theta} |x - y|^{- \theta}.
        \end{align*}
    \end{proof}
    
    \begin{lemma}
    \label{lem:mutual-chain-1}
    Let $W_t$ be a Brownian motion in $\mathbb{R}^d$ and $0 = t_0 < t_1 < \dots < t_m$. Then, for any $0 \le \theta \le d$ and $0 = x_0, x_1 , \dots , x_m \in \mathbb{R}^d$,
    \begin{equation}
    \label{eq:mutual-chain-delta}
        \mathbb{E} \left[ \prod_{i=1}^m \delta(W (t_i) - x_i) \right] \le C^m \prod_{i=1}^m |t_i - t_{i-1}|^{-(d-\theta)/2} \prod_{i=1}^m |x_i - x_{i-1}|^{-\theta}.
    \end{equation}
    Similarly, if $k < d$ and $0 \le \theta \le k/2$,
    \begin{equation}
    \label{eq:mutual-chain-riesz}
        \mathbb{E} \prod_{i=1}^m |W (t_i) - x_i|^{-k} \le C^m \prod_{i=1}^m |t_i - t_{i-1}|^{-(k-\theta)/2} \prod_{i=1}^m |x_i - x_{i-1}|^{-\theta}.
    \end{equation}
    \end{lemma}

    \begin{proof}
        Conditioned on $W[0, t_{m-1}]$, $W(t_m)$ is distributed as a Brownian motion starting at $W(t_{m-1})$ run for time $t_m - t_{m-1}$. Hence by \eqref{eq:delta-ineq}, 
        \[  
        \mathbb{E}\left[ \delta(W(t_m) - x_m) | W(t_{m-1}) \right] \le |t_m - t_{m-1}|^{-(d-\theta)/2} |W(t_{m-1}) - x_m|^{-\theta}
        \]
        and so
        \begin{align*}
            \mathbb{E} \left[ \prod_{i=1}^m \delta(W (t_i) - x_i) \right] &= \mathbb{E} \left[ \mathbb{E} \left[ \delta(W(t_m) - x_m) | W(t_{m-1}) \right]\prod_{i=1}^{m-1} \delta(W(t_i) - x_i)  \right] \\
            &\le C |t_m - t_{m-1}|^{-(d-\theta)/2} \mathbb{E} \left[ |W(t_{m-1}) - x_m|^{-\theta}\prod_{i=1}^{m-1} \delta(W(t_i) - x_i) \right].
        \end{align*}
        At this point, we condition on $W[0, t_{m-2}]$. Then by \eqref{eq:mutual-basic-delta}, we have
        \begin{align*}
            &\mathbb{E}\biggl[ \prod_{i=1}^{m-1} \delta(W(t_i) - x_i)| W(t_{m-1}) - x_m|^{-\theta} \biggr] \\
            &= \mathbb{E} \biggl[ \mathbb{E} \left[ \delta(W(t_{m-1}) - x_{m-1})| W(t_{m-1}) - x_m|^{-\theta}| W(t_{m-2}) \right]\prod_{i=1}^{m-2} \delta(W(t_i) - x_i)  \biggr] \\
            &\le C |t_m - t_{m-1}|^{-(d-\theta)/2} |t_{m-1} - t_{m-2}|^{-(d-\theta)/2} |x_m - x_{m-1}|^{-\theta} \mathbb{E} \biggl[ | W(t_{m-2}) - x_{m-1}|^{-\theta}\prod_{i=1}^{m-2} \delta(W(t_i) - x_i) \biggr].
        \end{align*}
        Since the expectations of the left and right hand sides have the same form, we may repeat this process $m-1$ times to obtain \eqref{eq:mutual-chain-delta}, namely
        \[  
        \mathbb{E} \left[ \prod_{i=1}^m \delta(W (t_i) - x_i) \right] \le C^m \prod_{i=1}^m |t_i - t_{i-1}|^{-(d-\theta)/2} \prod_{i=1}^m |x_i - x_{i-1}|^{-\theta}.
        \]
        The proof for \eqref{eq:mutual-chain-riesz} is almost identical, except that we use \eqref{eq:mutual-basic-riesz-2} instead of \eqref{eq:mutual-basic-delta}.
    \end{proof}

    \begin{lemma}
    \label{lem:mutual-chain-2}
        Let $W_t$ be a Brownian motion in $\mathbb{R}^d$. Given $t_0 = 0$ and $t_1 , t_2 , \dots , t_m > 0$, let $\sigma \in S_m$ be a permutation such that $t_{\sigma(1)} \le t_{\sigma(2)} \le \dots \le t_{\sigma(m)}$ and set $\sigma(0) = 0$. For any $k < d$,
    \begin{equation}
        \mathbb{E} \left[ \prod_{i=1}^m |W(t_i) - W(t_{i-1})|^{-k} \right] \le C^m \prod_{i=1}^m |t_{\sigma(i)} - t_{\sigma(i-1)}|^{-k/2}.
    \end{equation}
    \end{lemma}

    \begin{proof}
       We condition on $W[0, t_{\sigma(m-1)}]$. Under such conditioning, $W(t_{\sigma(m)})$ is distributed as a Gaussian with mean $W(t_{\sigma(m-1)})$ and variance $t_{\sigma(m)} - t_{\sigma(m-1)}$. If $\sigma(m) = m$, we may use \eqref{eq:riesz-ineq} to show that
       \begin{align*}
       \mathbb{E}\left[ \prod_{i=1}^m |W(t_i) - W(t_{i-1})|^{-k} \right] &= \mathbb{E}\left[ \prod_{i=1}^{m-1} |W(t_i) - W(t_{i-1})|^{-k} \mathbb{E} \left[ |W(t_m) - W(t_{m-1})| W(t_{m-1})\right] \right] \\
       &\le C|t_{\sigma(m)} - t_{\sigma(m-1)}|^{-k/2} \mathbb{E}\left[ \prod_{i=1}^{m-1} |W(t_i) - W(t_{i-1})|^{-k} \right].
       \end{align*}
       Otherwise, if $\sigma(m) = i_0$ for some $1 < i_0 < m$, the same conditioning argument and equation~\eqref{eq:mutual-basic-riesz-1} yields
       \begin{align*}
       &\mathbb{E}\biggl[ \prod_{i=1}^m |W(t_i) - W(t_{i-1})|^{-k} \biggr] \\
       &= \mathbb{E} \bigg[ \mathbb{E}\Big[|W(t_{i_0}) - W(t_{i_0 - 1})|^{-k} |W(t_{i_0}) - W(t_{i_0 + 1})|^{-k} \Big| W(t_{i_0 - 1}), W(t_{i_0 + 1}) \Big]\prod_{i \ne i_0, i_0 + 1} |W(t_i) - W(t_{i-1})|^{-k}\biggr]\\
       &\le C|t_{\sigma(m)} - t_{\sigma(m-1)}|^{-k/2} \mathbb{E}\biggl[ |W(t_{i_0 + 1}) - W(t_{i_0 - 1})|^{-k} \prod_{i \ne i_0, i_0 + 1} |W(t_i) - W(t_{i-1})|^{-k} \biggr].
       \end{align*}
       In either case, the problem is reduced to the same statement with $m-1$ in the place of $m$. Repeating $m$ times gives us the desired result.
    \end{proof}

\begin{lemma}
    \label{lem:mutual-approx-moment}
    Let $W^1, W^2, \dots , W^p$ be independent Brownian motions in $\mathbb{R}^d$ such that $d(p-1) < 2p$. For sufficiently small $\theta > 0$ and any bounded measurable $f$, we have
        \[
        \left| \mathbb{E} \left[ \int_{\mathbb{R}^d} \int_{[0, t]^p} f(x) \Delta_{\epsilon}(W^p(s^p) - x) \prod_{j=1}^{p-1} \delta(W^j(s^j) - x) \mathrm{d}\mathbf{s} \mathrm{d} x \right]^m \right| \le C^m \| f \|_{\sup}^m \epsilon^{\theta m} (m!)^p \left( \frac{t^m}{m!} \right)^{\frac{2p - d(p-1)}{2} - \theta}.
        \]
\end{lemma}

\begin{proof}
    The left hand side may be written as
    \begin{align*}
        \left| \mathbb{E} \left[\int_{[0, t]^p} \right. \right.& \left. \left.f(W^1(s^1)) \Delta_{\epsilon}(W^p(s^p) - W^1(s^1)) \prod_{j=2}^{p-1} \delta(W^j(s^j) - W^1(s^1)) \mathrm{d} \mathbf{s} \right]^m \right|\\
        &=  \left| \int_{[0, t]^{mp}}\mathbb{E} \left[ \prod_{i=1}^m f(W^1 (s^1_i)) \Delta_{\epsilon}(W^p(s^p_i) - W^1(s^1_i)) \prod_{j=2}^{p-1} \delta(W^j(s^j_i) - W^1(s^1_i)) \right] \mathrm{d} \mathbf{s} \right| \\
        &\le \int_{[0, t]^{mp}} \left|\mathbb{E} \left[ \prod_{i=1}^m f(W^1 (s^1_i)) \Delta_{\epsilon}(W^p(s^p_m) - W^1(s^1_m)) \prod_{j=2}^{p-1} \delta(W^j(s^j_i) - W^1(s^1_i)) \right] \right| \mathrm{d} \mathbf{s}.
    \end{align*}
    Without loss of generality assume $s^p_1 \le s^p_2 \le \dots \le s^p_m$ and let $\sigma^1 , \dots , \sigma^{p-1} \in S_m$ such that $\{s^{j}_{\sigma^j(i)} \}_{i=1}^m$ is increasing for each $1 \le j \le p-1$. We trisect each interval $[s^p_{i-1}, s^p_i]$ into thirds and label $s^p_{i \pm 1/3} := s^p_i + (s^p_{i \pm 1} - s^p_i)/3$. If we condition on $\{s^p_{i \pm 1/3}: 1 \le i \le m\}$, the points $W(s^p_i)$ become independent Gaussians with mean
    \[
    \overline{W}(s_i^p) := W(s^p_{i-1/3}) + \frac{s^p_i- s^p_{i - 1/3}}{s^p_{i + 1/3} - s^p_{i - 1/3}} (W(s^p_{i+1/3}) - W(s^p_{i - 1/3}))
    \]
    and variance
    \[
    \frac{(s^p_{i + 1/3} - s^p_i)(s^p_i - s^p_{i - 1/3})}{s^p_{i + 1/3} - s^p_{i - 1/3}} = \left( \frac{1}{s^p_{i+1/3} - s^p_i} + \frac{1}{s^p_i - s^p_{i-1/3}} \right)^{-1} = \frac{1}{3}\left( \frac{1}{s^p_{i+1} - s^p_i} + \frac{1}{s^p_i - s^p_{i-1}} \right)^{-1}.
    \]
    We now condition on the set $\{W^j[0, t] :1 \le j \le p-1 \} \cup \{ W^p(s^p_{i \pm 1/3}) : 1 \le i \le m \}$. This gives us
    \begin{align*}
        \Biggl|\mathbb{E} \bigg[ &\prod_{i=1}^m f(W^1 (s^1_i))\Delta_{\epsilon}(W^p(s^p_m) - W^1(s^1_m)) \prod_{j=2}^{p-1} \delta(W^j(s^j_i) - W^1(s^1_i)) \biggr] \Biggr| \\
        &= \Biggl|\mathbb{E} \biggl[ \prod_{i=1}^m f(W^1 (s^1_i)) \mathbb{E}\Bigl[\Delta_{\epsilon}(W^p(s^p_i) - W^1(s^1_i)) \Big| W^1(s^1_i), W^p(s^p_{i \pm 1/3})\Bigr]\prod_{j=2}^{p-1} \delta(W^j(s^j_i) - W^1(s^1_i)) \biggr] \Biggr| \\
        &\le \mathbb{E} \Biggl[ \prod_{i=1}^m |f(W^1 (s^1_i))|\biggl| \mathbb{E}\Bigl[\Delta_{\epsilon}(W^p(s^p_i) - W^1(s^1_i)) \Big| W^1(s^1_i), W^p(s^p_{i \pm 1/3})\Bigr]\biggr| \prod_{j=2}^{p-1} \delta(W^j(s^j_i) - W^1(s^1_i)) \Biggr] \\
        &\le C^m \| f \|_{\sup}^m \epsilon^{\theta m} \prod_{i=1}^m \Bigl( |s^p_i - s^p_{i-1}|^{-\frac{(p+1)\theta}{2}} + |s^p_{i+1} - s^p_i |^{-\frac{(p+1)\theta}{2}} \Bigl) \\
        &\quad \quad \quad \times \mathbb{E} \biggl[ \prod_{i=1}^m |\overline{W}^p(s^p_i) - W^1(s^1_i)|^{-d + (p-1)\theta} \prod_{j=2}^{p-1} \delta(W^j(s^j_i) - W^1(s^1_i))\biggr]
    \end{align*}
    The last line comes from the inequality
    \[  
    |\mathbb{E} \Delta_{\epsilon}(W_t - x)| \le C \epsilon^{\theta} t^{-(p+1)\theta/2} |x|^{-d + (p-1)\theta},
    \]
    which is a consequence of applying H\"{o}lder's inequality to \eqref{eq:approx-ineq} with weights $ \frac{(p+1)\theta}{d}, \frac{d - (p+d+1)\theta}{d}, 0, \theta$. We will always choose $\theta$ such that $d > (p + d + 1) \theta$.
    
    Now we condition on $W^1[0, t]$. By the independence of $W^2, \dots , W^p$, we may distribute the expectation over the product and apply Lemma~\ref{lem:mutual-chain-1} to get
    \begin{align*}
        \mathbb{E} \biggl[ \prod_{i=1}^m &|\overline{W}^p(s^p_i) - W^1(s^1_i)|^{-d + (p-1)\theta} \prod_{j=2}^{p-1} \delta(W^j(s^j_i) - W^1(s^1_i))\biggr] \\
        &= \mathbb{E} \Biggl[ \mathbb{E} \biggl[\prod_{i=1}^m |\overline{W}^p(s^p_i) - W^1(s^1_i)|^{-d + (p-1)\theta} \Big| W^1[0, t] \biggr] \prod_{j=2}^{p-1}\mathbb{E} \biggl[\prod_{i=1}^m\delta(W^j(s^j_i) - W^1(s^1_i))\Big| W^1[0, t] \biggr] \Biggr] \\
        &\le C^m \prod_{j=2}^{p-1}\prod_{i=1}^m |s^j_{\sigma^j(i)} - s^j_{\sigma^j(i-1)}|^{-(d-\theta)/2} \\
        &\quad \quad \times \mathbb{E}\biggl[ \prod_{i=1}^m |\overline{W}^p(s^p_i) - W^1(s^1_i)|^{-d + (p-1)\theta} \prod_{j=2}^{p-1} |W^1(s^1_{\sigma^j(i)}) - W^1(s^1_{\sigma^j(i-1)})|^{-\theta} \biggr].
    \end{align*}
    By H\"{o}lder's inequality, the last term is bounded by
    \begin{multline*}
    \mathbb{E}\left[ \prod_{i=1}^m |W^p(s^p_i) - W^1(s^1_i)|^{-d +(p-1)\theta} \prod_{j=2}^{p-1} \prod_{i=1}^m |W^1(s^1_{\sigma^j(i)}) - W^1(s^1_{\sigma^j(i-1)})|^{-\theta} \right] \\
    \le \mathbb{E}\left[ \prod_{i=1}^m |W^p(s^p_i) - W^1(s^1_i)|^{-d + \theta} \right]^{\frac{d - (p-1)\theta}{d - \theta}} \prod_{j=2}^{p-1} \mathbb{E} \left[ \prod_{i=1}^m |W^1(s^1_{\sigma^j(i)}) - W^1(s^1_{\sigma^j(i-1)})|^{-d + \theta}\right]^{\frac{\theta}{d - \theta}}.
    \end{multline*}
    The second term is bounded by Lemma~\ref{lem:mutual-chain-2},
    \[  
    \mathbb{E} \left[ \prod_{i=1}^m |W^1(s^1_{\sigma^j(i)}) - W^1(s^1_{\sigma^j(i-1)})|^{-d + \theta}\right] \le \prod_{i=1}^m |s^1_{\sigma^1(i)} -s^1_{\sigma^1(i-1)}|^{-(d-\theta)/2}.
    \]
    As for the first term, we first apply Lemma~\ref{lem:mutual-chain-1} to obtain
    \[
    \mathbb{E}\left[ \prod_{i=1}^m |W^p(s^p_i) - W^1(s^1_i)|^{-d + \theta} \right] \le \prod_{i=1}^m |s^1_{\sigma^1(i)} - s^1_{\sigma^1(i-1)}|^{-(d - \theta)/4} \mathbb{E} \left[ \prod_{i=1}^m |\overline{W}^p(s^p_{\sigma^1(1)}) - \overline{W}^p(s^p_{\sigma^1(i-1)}) |^{-(d - \theta)/2}\right]
    \]
    The right-hand side is very similar to Lemma~\ref{lem:mutual-chain-2}, except that we have $\overline{W}^p(s^p_i)$ instead of $W^p(s^p_i)$. However, an observation of the proof quickly reveals that the identical proof gives the same bound of
    \[  
    \mathbb{E} \left[ \prod_{i=1}^m |\overline{W}^p(s^p_{\sigma^1(1)}) - \overline{W}^p(s^p_{\sigma^1(i-1)}) |^{-(d - \theta)/2}\right] \le C^m \prod_{i=1}^m |s^p_i - s^p_{i-1}|^{-(d-\theta)/4}.
    \]
    Combining all the above, we have
    \begin{multline*}
        \left|\mathbb{E} \left[ \prod_{i=1}^m f(W^1 (s^1_i) \Delta_{\epsilon}(W^p(s^p_m) - W^1(s^1_m)) \prod_{j=2}^{p-1} \delta(W^j(s^j_i) - W^1(s^1_i)) \right] \right| \\
        \le C^m \epsilon^{\theta m} \prod_{i=1}^m |s^1_{\sigma^1(i)} - s^1_{\sigma^1(i-1)}|^{-\frac{d- (p-1)\theta}{4} -\frac{(p-2)\theta}{2}} \prod_{j=2}^{p-1} \prod_{i=1}^m |s^j_{\sigma^1(i)} - s^j_{\sigma^j(i-1)}|^{-(d - \theta)/2} \\
        \times \prod_{i=1}^m |s^p_i - s^p_{i-1} |^{-\frac{d- (p-1)\theta}{4}}   \left( |s^p_i - s^p_{i-1} |^{-(p+1)\theta/2} + |s^p_{i+1} - s^p_i|^{-(p+1)\theta/2}\right).
    \end{multline*}
    Note that when $d = 3$, we only consider $p = 2$ so there are no terms of the form $|s^j_{\sigma^j(i)} - s^j_{\sigma^j(i-1)}|^{-(d-\theta)/2}$. As such, we can always choose sufficiently small $\theta$ so that the right-hand side is integrable. The proof is completed by integrating both sides over $([0, t]_<^m)^p$ and then summing over all $(m!)^p$ possible orderings.
\end{proof}

\section{LDP for occupation measures: the Mukherjee-Varadhan topology}
\label{sec:MV-top}

In this section, we define the analog of the Mukherjee-Varadhan topology for $p$-fold product measures and prove the LDP for the occupation measures in this topology. The methods and proofs are similar to the original works of Mukherjee and Varadhan \cite{MukherjeeVaradhan2016,Mukherjee2017}, but there is one key difference in how we generalize to joint measures.

\subsection{Compactification of $\widetilde{\mathcal{M}}_1^{\otimes p}(\mathbb{R}^d)$}
\label{subsec:MV-top-def}

Recall the set $\widetilde{\mathcal{X}}_{\le 1}^{\otimes p}(\mathbb{R}^d)$ of \eqref{eq:MV-bdd-set},
\[
    \widetilde{\mathcal{X}}_{\le 1}^{\otimes p}(\mathbb{R}^d) = \left\{ \xi^{\otimes p} = \{ \widetilde{\alpha}_i^{\otimes p} \}_{i \in I} : \widetilde{\alpha}_i^j \in \widetilde{\mathcal{M}}_{\le 1} (\mathbb{R}^d), \sum_{i \in I} \alpha_i^j(\mathbb{R}^d) \le 1 \text{ for all }j, \; I \text{ is at most countable} \right\}.
\]
Clearly, $\widetilde{\mathcal{M}}_1^{\otimes p}$ is a subset of $\widetilde{\mathcal{X}}_{\le 1}^{\otimes p}$ via the inclusion $\widetilde{\alpha} \mapsto \{ \widetilde{\alpha}\}$. In fact, we proceed to show that $\widetilde{\mathcal{M}}_1^{\otimes p}$ is actually a topological subspace of $\widetilde{\mathcal{X}}_{\le 1}^{\otimes p}$, i.e., that the inclusion map is a homeomorphism onto its image. We can also view $\widetilde{\mathcal{X}}_{\le 1}^{\otimes p}(\mathbb{R})$ as a subset of 
\begin{equation}
\label{eq:MV-set}
    \widetilde{\mathcal{X}}^{\otimes p}(\mathbb{R}^d) = \left\{ \xi^{\otimes p} = \{ \widetilde{\alpha}_i^{\otimes p} \}_{i \in I} : \widetilde{\alpha}_i^{\otimes p} \in \widetilde{\mathcal{M}}^{\otimes p}(\mathbb{R}^d), \; \sum_{i \in I} \alpha_i^j(\mathbb{R}^d) < \infty, \;\; I \text{ is at most countable} \right\}.
\end{equation}
In particular, $\emptyset \in \widetilde{\mathcal{X}}^{\otimes p}$. The sets $\{ \widetilde{\alpha}_i \}_{i \in I}$ are unordered but allowed to repeat, i.e., they should be seen as multisets. The zero tuple $(0, \dots , 0)$ is not allowed, as they should be erased to remove redundancy. On the other hand, $\alpha^{\otimes p} = (\alpha^1, \dots , \alpha^p)$ is allowed to contain zero measures $\alpha^j = 0$ as part of its tuple, as long as not all of them are zero.

\begin{remark*}
    The papers \cite{Mukherjee2017} and \cite{ErhardPoisat2026} view $\widetilde{\alpha}^{\otimes p}$ as a product measures in $\mathbb{R}^{dp}$ rather than a $p$-tuple of measures in $\mathbb{R}^d$. Because the $\alpha^j$'s are not necessarily probability measures, this perspective loses some information about the marginals. In particular, this means that elements with zero measures as part of its tuple get ignored (they become the zero measure). We have altered the definition in the way presented above, as we feel this is the more natural generalization of \cite{MukherjeeVaradhan2016} (see the remarks following Example~\ref{ex:profile-decomposition} and Lemma~\ref{lem:projection}).
\end{remark*}

We define a topology on $\widetilde{\mathcal{X}}^{\otimes p}$ through a class of test functions. For any integer $k \ge 1$, define $\mathcal{F}_k(\mathbb{R}^d)$ to be the set of functions $f : (\mathbb{R}^d)^p \to \mathbb{R}$ that is continuous, diagonally shift-invariant in the sense that
\[
f(x_1 + x , x_2 + x , \dots , x_k + x) = f(x_1 , x_2 , \dots , x_k) \quad \text{for all } x_i, x \in \mathbb{R}^d,
\]
and vanishing at infinity, i.e.,
\[  
f(x_1 , x_2 , \dots , x_k) \to 0 \quad \text{as} \quad \max_{i_1 \ne i_2} |x_{i_1} - x_{i_2}| \to \infty.
\]
Note that $\mathcal{F}_1(\mathbb{R}^d)$ only contains the zero map. For a multi-index $\mathbf{k} = (k_1 , \dots , k_p)$ where $k_1, \dots , k_p \ge 0$ are integers, define $| \mathbf{k}| = \sum k_i$ and
\[  
\Lambda_{\mathbf{k}}( f, \xi^{\otimes p}) := \sum_{i \in I} \int f(x_1^1, \dots ,x_{k_1}^1, \dots , x_{k_p}^p) \mathrm{d}\alpha_i^1(x_1^1) \dots \mathrm{d} \alpha_i^1(x_{k_1}^1) \dots \mathrm{d}\alpha_i^p(x_{k_p}^p)
\]
for any $f \in \mathcal{F}_{|\mathbf{k}|}$. We may sometimes omit the subscript when $\mathbf{k}  =(1, \dots,  1)$. This map is well-defined since the integral is diagonally shift-invariant. We equip $\widetilde{\mathcal{X}}^{\otimes p}$ with the weakest topology that is continuous for all $\Lambda^{\otimes p}_{\mathbf{k}}(f , \cdot)$. Since $\mathcal{F}_{|\mathbf{k}|}$ is separable, we may metrize $\widetilde{\mathcal{X}}^{\otimes p}$ with the (pseudo)metric
\begin{equation}
    \label{eq:MV-metric}
    \mathbf{D}^{\otimes p}(\xi^{\otimes p}, \zeta^{\otimes p}) := \sum_{ \mathbf{k}} \sum_{r = 1}^{\infty} \frac{1}{(2p)^{|\mathbf{k}|}} \frac{2^{-r}}{ | \mathbf{k}|(1 + \| f_{k, r} \|_{\sup})} \left| \Lambda_{\mathbf{k}} (f_{k, r}, \xi^{\otimes p}) - \Lambda_{\mathbf{k}} (f_{k, r}, \zeta^{\otimes p}) \right|
\end{equation}
where $\{ f_{k, r} \}_{r \ge 1}$ is a dense subset of $\mathcal{F}_{k}$. We show that $(\widetilde{\mathcal{X}}_{\le 1}^{\otimes p}, \mathbf{D}^{\otimes p})$ is a compact metric space that contains $\widetilde{\mathcal{M}}_1^{\otimes p}$ as a dense subspace.

\begin{lemma}
    $\mathbf{D}^{\otimes p}$ is a metric on $\widetilde{\mathcal{X}}$.
\end{lemma}

\begin{proof}
    Symmetry, positivity, and the triangle inequality are trivial, so we only show that $\mathbf{D}^{\otimes p}(\xi_1, \xi_2) = 0$ implies $\xi_1 = \xi_2$. That is, we show that $\xi \in \widetilde{\mathcal{X}}^{\otimes p}$ is uniquely determined by the values of $\Lambda^{\otimes p}(f, \xi)$. Our general strategy is to take large values of $\mathbf{k}$ and use the law of large numbers on the empirical distributions to retrieve $\xi^{\otimes p}$.
    
    To this end, fix some $\xi^{\otimes p} = \{ \widetilde{\alpha}_i^{\otimes p} \}_{i \in I} \in \widetilde{\mathcal{X}}^{\otimes p}$. For each $i \in I$, denote the mass of $\alpha_i^j$ by $m_i^j = \alpha_i^j(\mathbb{R}^d)$ and its renormalized (probability) measure as $\beta_i^j := (m_i^j)^{-1}\alpha_i^j$. We sample $X_{i, 0}^j, X_{i, 1}^j, \dots$ independently from $\beta_i^j$. We use the shorthand $m_i^{\mathbf{k}} = \prod_{j=1}^p (m_i^j)^{k_j}$ and $m^{\mathbf{k}} = \sum_{i \in I} m_i^{\mathbf{k}}$.
    
    Now we describe the integrals $\Lambda(f, \cdot)$ in terms of $X_{i, k}^j$. For any $f \in C_0(\mathbb{R}^{d| \mathbf{k}|})$, define $\widetilde{f} \in \mathcal{F}_{|\mathbf{k}| + 1} (\mathbb{R}^d)$ as $\widetilde{f}(x_0 \dots ,x_k) = f(x_1 - x_0, \dots ,x_k -x_0)$ ($\widetilde{f}$ is indeed diagonally shift-invariant and vanishes at infinity). Then,
    \begin{align*}
    \Lambda_{\mathbf{k} + (1, 0,\dots , 0)} (\widetilde{f}, \xi^{\otimes p}) = \sum_{i \in I} \Lambda_{\mathbf{k} + (1, 0,\dots , 0)}( \widetilde{f}, \widetilde{\alpha}_i^{\otimes p})&= \sum_{i \in I} m_i^1 m_i^{\mathbf{k}} \Lambda_{\mathbf{k} + (1, 0,\dots , 0)} ( \widetilde{f}, \widetilde{\beta}_i^{\otimes p} ) \\
    &= \sum_{i \in I} m_i^1 m_i^{\mathbf{k}} \mathbb{E} f(X_{i, 1}^1 - X_{i, 0}^1, \dots , X_{i, k_p}^{p} - X_{i, 0}^1) \\
    &= m^{\mathbf{k} + (1, 0, \dots, 0)}\mathbb{E} f(Y_1^1 - Y_0^p, \dots , Y_{k_p}^p - Y_0^p ),
    \end{align*}
    where $(Y_0^1, Y_1^1, \dots , Y_{k_1}^1 , \dots , Y_{k_p}^p)$ is distributed as
    \[  
    (Y_0^1, \dots , Y_{k_1}^1 , \dots , Y^p_{k_p}) = (X_{i, 0}^1, \dots , X_{i, k_1}^1 , \dots , X_{i, k_p}^p) \text{ with probability } \frac{m_i^{\mathbf{k} + (1, 0, \dots, 0)}}{m^{\mathbf{k} + (1, 0, \dots, 0)}} = \frac{m_i^1 \prod_{j=1}^p (m_i^j)^{k_j}}{\sum\limits_{i' \in I} m_{i'}^1\prod_{j=1}^p  (m_{i'}^j)^{k_j}}.
    \]
    
    In other words, we may test $(Y_1^1 - Y_0^1 , \dots Y_{k_p}^p - Y_0^p)$ against arbitrary functions in $C_0(\mathbb{R}^{d(k-1)})$. Hence we retrieve its law completely, along with the value of $m^{\mathbf{k}}$. Since we know $m^{\mathbf{k}}$ for any $\mathbf{k}$, we can retrieve the multiset $\{(m_i^1 , \dots, m_i^p)\}_{i \in I}$ via the method of moments. The case where $m_i^1 = 0$ is handled by replacing $X_0^1$ with some other $X_0^j$.

    It only remains to determine the measures $\beta_i^j$. To this end, order the set $\{(m_i^1 , \dots, m_i^p)\}_{i \in I}$ in lexicographically decreasing order, and suppose there are $r$ elements tied for first (there can only be finitely many, as $\sum_i m_i^j$ is finite). We may choose a sequence $\mathbf{k}^n$ of multi-indices such that each $k_j^n$ diverges to infinity (unless $m_i^j = 0$, in which case we take $k_j^n = 0$) and $m_1^{\mathbf{k}}, \dots , m_r^{\mathbf{k}}$ dominates all other terms. Then for large $n$, $Y^{\mathbf{k}}$ is very close to the uniform distribution on $\{ X_1^{\mathbf{k}^n}, \dots, X_r^{\mathbf{k}^n} \}$. Now if we compute the law of
    \[  
    \left( \frac{\sum_{k=1}^{k_1} Y_{k}^1}{k_1} - Y_0^1 , \dots , \frac{\sum_{k=1}^{k_p} Y_{k}^p}{k_p} - Y_0^1\right),
    \]
    which is possible since it only depends on the differences $Y_k^j - Y_0^1$, it will converge to the uniform distribution on $\{ \widetilde{\beta}_1^{\otimes p}, \dots , \widetilde{\beta}_r^{\otimes p} \}$ in $\widetilde{\mathcal{M}}_1^{\otimes p}$. Since we also know the number $r$, this uniquely determines $\widetilde{\alpha}_1^{\otimes p}, \dots , \widetilde{\alpha}_r^{\otimes p}$. By removing $\widetilde{\alpha}_1^{\otimes p}, \dots , \widetilde{\alpha}_r^{\otimes p}$ and repeating this process (possibly infinitely many times), we obtain the entire set $\{ \widetilde{\alpha}_i^{\otimes p} \}$.
\end{proof}

\begin{lemma}
\label{lem:decomposition-convergence}
    A sequence $\widetilde{\mu}^{\otimes p}$ in $\widetilde{\mathcal{M}}_1^{\otimes p}$ converges to $\{ \widetilde{\alpha}_i^{\otimes p} \}$ in $(\widetilde{\mathcal{X}}^{\otimes p}, \mathbf{D}^{\otimes p})$ if there exists a decomposition
    \[
    \mu_n^{\otimes p} = \sum_{i \in I} \alpha_{i, n}^{\otimes p} + \beta_n^{\otimes p}
    \]
    and points $x_{i, n} \in \mathbb{R}^d$ such that
    \begin{enumerate}
        \item $\alpha_{i, n}^j \ast \delta_{x_i^n} \to \alpha_i^j$ weakly for all $i \in I$ and  $j = 1 , \dots , p$,
        \item $\beta_n^j$ \emph{totally disintegrates}, i.e., $\lim_{n \to \infty} \sup_{x \in \mathbb{R}^d} \beta_n^j (B(x, R)) = 0$ for any finite $R > 0$.
        \item Distinct sequences are \emph{widely separated}, i.e., $\lim_{n \to \infty} |x_{i_1}^n - x_{i_2}^n| = \infty$ for any $i_1 \ne i_2$.
    \end{enumerate}
    As a consequence, the inclusion $\widetilde{\mathcal{M}}_1 \hookrightarrow \widetilde{\mathcal{X}}_{\le 1}^{\otimes p}$ is a continuous injection and $\widetilde{\mathcal{M}}_1$ is a dense subset of $\widetilde{\mathcal{X}}_{\le 1}^{\otimes p}$.
\end{lemma}

We call this the \emph{profile decomposition} of $\mu_n$, following the name used in the literature when $p=1$ (e.g., see \cite[Section~4.5]{Tao2013}).

\begin{remark*}
    Summation of tuples are done entry-wise, i.e., $\alpha_1^{\otimes p} + \alpha_2^{\otimes p} = (\alpha_1^1 + \alpha_2^1, \dots, \alpha_1^p + \alpha_2^p)$. This operation does not behave well under the equivalence relation, in the sense that the sum depends on the representative chosen and so $\widetilde{\alpha}_1^{\otimes p} + \widetilde{\alpha}_2^{\otimes p}$ is not well-defined. Nevertheless, the above lemma is valid since we give additional information on the shift operations.
\end{remark*}

\begin{proof}
    It suffices to show that $\Lambda_{\mathbf{k}}( f, \widetilde{\mu}_n^{\otimes p})$ converges to $\sum_{i \in I} \Lambda_{\mathbf{k}} (\widetilde{\alpha}_i^{\otimes p})$ for any $f \in \mathcal{F}_{|\mathbf{k}|}$. For simplicity, suppose $p = 1$, $|I| = 2$, and $k = 2$. That is, we have the decomposition $\mu_n = \alpha_{1, n} + \alpha_{2, n} + \beta_n$ satisfying the properties (i)-(iii). We see that for any $f \in \mathcal{F}_2$,
    \begin{align*}
        \Lambda_2( f, \mu_n ) &= \int f(x^1, x^2)  (\alpha_{1, n} + \alpha_{2, n} + \beta_n)(\mathrm{d} x^1) (\alpha_{1, n} + \alpha_{2, n} + \beta_n)(\mathrm{d} x^2) \\
        &= \int f(x^1, x^2) \alpha_{1, n} (\mathrm{d} x^1) \alpha_{1, n}(\mathrm{d} x^2) + \int f(x^1, x^2) \alpha_{2, n} (\mathrm{d} x^1) \alpha_{2, n}(\mathrm{d} x^2) \\
        & \qquad + \int f(x^1, x^2) \alpha_{1, n} (\mathrm{d} x^1) \alpha_{2, n}(\mathrm{d} x^2) + \int f(x^1, x^2) \alpha_{2, n} (\mathrm{d} x^1) \alpha_{1, n}(\mathrm{d} x^2) \\
        & \qquad + \int f(x^1, x^2) (\alpha_{1, n} + \alpha_{2, n} + \beta_n)(\mathrm{d} x^1) \beta_n(\mathrm{d} x^2) + \int f(x^1, x^2) \beta_n(\mathrm{d} x^1) (\alpha_{1, n} + \alpha_{2, n})(\mathrm{d} x^2).
    \end{align*}
    The first line equals $\Lambda_2( f, \alpha_{1, n}) + \Lambda_2( f, \alpha_{2, n} )$, which converges to $\Lambda_2( f, \alpha_1)+ \Lambda_2( f, \alpha_2 )= \Lambda_2( f, \{ \widetilde{\alpha}_1 ,\widetilde{\alpha}_2\} )$ by the properties of weak convergence. Meanwhile, the cross-terms of the second line converge to zero since $\alpha_{1, n}$ and $\alpha_{2, n}$ are widely separated, while the third line goes to zero since $\beta_n$ totally disintegrates---see \cite{MukherjeeVaradhan2016} for a detailed proof.

    The above proof easily generalizes to all $p$, $I$, and $\mathbf{k}$. Indeed, we can use the same decomposition to split $\Lambda_{\mathbf{k}}(f, \mu_n^{\otimes p})$ into a sum of integrals, and the only terms that survive are the ones only containing $\alpha_{i, n}^j$'s with the same drift. Since both sides are uniformly bounded by $\| f \|_{\sup}^k$ since $\mu_n^j(\mathbb{R}^d) = 1$, summing over infinitely many terms is not a problem.

    Now suppose $\widetilde{\mu}_n^{\otimes p} \to \widetilde{\mu}^{\otimes p}$ weakly (i.e., in $\widetilde{\mathcal{M}}_1^{\otimes p}$). Since $\widetilde{\mathcal{M}}_1^{\otimes p}$ is a quotient under a continuous group action, this implies that there is a sequence $\{x_n\}$ such that $\mu_n^j \ast \delta_{x_n} \to \mu^j$ weakly for each $j$. In other words, $\mu_n^{\otimes p}$ is its own profile decomposition and hence $\mathbf{D}^{\otimes p} (\widetilde{\mu}_n^{\otimes p} , \widetilde{\mu}^{\otimes p}) \to 0$.
    
    To see that $\widetilde{\mathcal{M}}_1^{\otimes p}$ is dense in $\widetilde{\mathcal{X}}_{\le 1}^{\otimes p}$, we simply construct $\alpha_{i, n}^j = \alpha_i^j \ast \delta_{x_i^n}$ for some choice of $x_i^n$ satisfying condition (iii). We also choose $\beta_n$ to have total mass $\beta_n^j(\mathbb{R}^d) = 1 - \sum_i \alpha_{i, n}^j(\mathbb{R}^d)$ with sufficient spread (e.g., take a Gaussian with variance $n$) so that it totally disintegrates. Now it is easy to see that the sequence $\mu_n^{\otimes p}$ with marginals
    \begin{equation}  
    \label{eq:profile-decomposition}
    \mu_n^j = \sum_{i \in I} \alpha_i^j \ast \delta_{x_i^n} + \beta_n
    \end{equation}
    satisfies the conditions of the lemma and hence converges to $\{ \widetilde{\alpha}_i^{\otimes p} \}_{i \in I}$.
\end{proof}

Next we prove that $\widetilde{\mathcal{X}}_{\le 1}^{\otimes p}$ is compact. Most of the work is done for us by concentration-compactification criterion. The following version is stated in \cite[Theorem~4.5.4]{Tao2013} (slightly rephrased), and \cite{MukherjeeVaradhan2016} also includes an equivalent statement in the proof of their Theorem~3.2.

\begin{lemma}[{\cite[Theorem~4.5.4]{Tao2013}}]
    \label{lem:profile-decomposition}
    Let $\mu_n$ be a sequence of Borel probability measures on $\mathbb{R}^d$. Then, after passing to a subsequence, $\mu_n$ admits a profile decomposition
    \[  
    \mu_n = \sum_{i \in I} \alpha_{i, n} + \beta_n.
    \]
    That is, the decomposition satisfies the conditions of Lemma~\ref{lem:decomposition-convergence}.
\end{lemma}

\begin{corollary}
\label{cor:profile-decomposition}
    Given any sequence $\mu_n^{\otimes p} \in (\mathcal{M}_1)^p$, there exists a subsequence with a profile decomposition.
\end{corollary}

\begin{proof}
    By Lemma~\ref{lem:profile-decomposition}, we may find a subsequence where each $\mu_n^j$ has a profile decomposition $\mu_n^j = \sum_{i \in I} \alpha_{i, n}^j + \beta_n^j$ (we can share the same index set $I$ simply by taking a disjoint union and allowing zero measures) with shifts $x_{i, n}^j$. Now, by passing to a further subsequence, we may assume that the differences of any pair $x_{i_1, n}^{j_1} - x_{i_2, n}^{j_2}$ is either convergent or diverges to infinity. Moreover, if $x_{i_1, n}^{j_1} - x_{i_2, n}^{j_2}$ converges to some $x \in \mathbb{R}^d$, then we may replace each $x_{i_2, n}^{j_2}$ with $x_{i_1, n}^{j_1}$ and $\alpha_{i_2}^{j_2}$ with $\alpha_{i_2}^{j_2} \ast \delta_x$ and still get a profile decomposition. Hence, we may assume two sequences $x_{i_1, n}^{j_1}$ and $x_{i_2, n}^{j_2}$ are either identical or diverge away from each other.

    Now group the measures which have the same shift (this is clearly an equivalence relation). If $\alpha_{i_1}^1, \dots , \alpha_{i_p}^p$ are in the same group (there cannot be two measures with the same $j$-index in the same group, by the definition of profile decomposition for $\mu_n^j$), take the equivalence class of $(\alpha_{i_1}^1 , \dots , \alpha_{i_p}^p)$ to be an element of $\xi^{\otimes p}$. If some of the $j$-indices are missing, fill them with zero measures and include the tuple in $\xi^{\otimes p}$. It is clear that $\xi^{\otimes p}$ along with $\beta_n^{\otimes p} = (\beta_n^1 , \dots , \beta_n^p)$ satisfies the conditions of Lemma~\ref{lem:decomposition-convergence}.
\end{proof}

\begin{example}
\label{ex:profile-decomposition}
    We give a concrete example of a profile decomposition and compare it to the topology of \cite{Mukherjee2017}. Suppose $p = 3$ and $\mu_n^1$, $\mu_n^2$ and $\mu_n^3$ have decompositions
    \begin{center}
    \label{fig:profile-decomposition}
    \begin{tikzpicture}
    
    \node (y1) at (0, 3) {$\mu_n^{1} =$};
    \node (a11) at (2, 3) {$\alpha_{1,n}^{1}$};
    \node (p1a) at (3, 3) {$+$};
    \node (a21) at (4, 3) {$\alpha_{2,n}^{1}$};
    \node (p1b) at (5, 3) {$+$};
    \node (a31) at (6, 3) {$\alpha_{3,n}^{1}$};
    \node (p1c) at (7, 3) {$+$};
    \node (b1)  at (8, 3) {$\beta_n^{1}$};
    
    \node (y1) at (0, 2) {$\mu_n^{2} =$};
    \node (a12) at (2, 2) {$\alpha_{1,n}^{2}$};
    \node (p2a) at (3, 2) {$+$};
    \node (a22) at (4, 2) {$\alpha_{2,n}^{2}$};
    \node (p2c) at (7, 2) {$+$};
    \node (b2)  at (8, 2) {$\beta_n^{2}$};
    
    \node (y1) at (0, 1) {$\mu_n^{3} =$};
    \node (a13) at (2, 1) {$\alpha_{1,n}^{3}$};
    \node (p3a) at (3, 1) {$+$};
    \node (a23) at (4, 1) {$\alpha_{2,n}^{3}$};
    \node (p3b) at (5, 1) {$+$};
    \node (a33) at (6, 1) {$\alpha_{3,n}^{3}$};
    \node (p3c) at (7, 1) {$+$};
    \node (b3)  at (8, 1) {$\beta_n^{3}$};
    
    \node[draw, rounded corners=18pt,
          fit=(a11) (a12), inner sep=8pt] (big) {};
    \node[draw, rounded corners=18pt,
      fit=(a21) (a22) (a23), inner sep=8pt, draw] (big) {};
    \end{tikzpicture}
    \end{center}
    where the circled groups indicate which drifts coalesce. Then, $\widetilde{\mu}_n^{\otimes p}$ converges to the set with representatives
    \[
    \{ (\alpha_{1}^1, \alpha_{1}^2, 0), (\alpha_{2}^1, \alpha_{2}^2, \alpha_{2}^3), (\alpha_{3}^1, 0, 0), (0, 0, \alpha_{1}^3), (0, 0, \alpha_{3}^3) \}.
    \]
    On the other hand, in the topologies of \cite{Mukherjee2017} and \cite{ErhardPoisat2026}, the limit would simply be $\alpha_2^1 \alpha_2^2 \alpha_2^3$. While this is fine for retrieving the intersection measure, it loses information about the marginal distributions. See Lemma~\ref{lem:projection} for additional benefits.
\end{example}

\begin{lemma}
    The space $\widetilde{\mathcal{X}}_{\le 1}^{\otimes p}$ is compact and contains $\widetilde{\mathcal{M}}_1^{\otimes p}$ as a topological subspace (i.e., $\mathbf{D}^{\otimes p}$ induces the quotient weak topology on $\widetilde{\mathcal{M}}_1^{\otimes p}$). Therefore, $\widetilde{\mathcal{X}}_{\le 1}^{\otimes p}$ is a compactification of $\widetilde{\mathcal{M}}_1^{\otimes p}$ and also its completion under the metric $\mathbf{D}^{\otimes p}$.
\end{lemma}

\begin{proof}
    By Lemma~\ref{lem:decomposition-convergence}, we know that $\widetilde{\mathcal{M}}_1^{\otimes p}$ is a dense subset of $\widetilde{\mathcal{X}}_{\le 1}^{\otimes p}$. Therefore, to show that $\widetilde{\mathcal{X}}_{\le 1}^{\otimes p}$ is compact, it suffices to show that any sequence in $\widetilde{\mathcal{M}}_1^{\otimes p}$ has a subsequence that converges to some $\xi \in \widetilde{\mathcal{X}}_{\le 1}^{\otimes p}$, which is immediate from Lemma~\ref{lem:decomposition-convergence} and Corollary~\ref{cor:profile-decomposition}.

    Now we show that $\mathbf{D}^{\otimes p}$ induces the quotient weak topology on $\widetilde{\mathcal{M}}_1^{\otimes p}$. We already know from Lemma~\ref{lem:decomposition-convergence} that the injection map is a continuous injection, so it only remains to show that $\mathbf{D}^{\otimes p}(\widetilde{\mu}_n^{\otimes p}, \widetilde{\mu}^{\otimes p}) \to 0$ implies $\widetilde{\mu}_n^{\otimes p} \to \widetilde{\mu}^{\otimes p}$ in the usual quotient weak topology of $\widetilde{\mathcal{M}}_1^{\otimes p}$. To this end, take any subsequence of $\mu_n^{\otimes p}$. By Corollary~\ref{cor:profile-decomposition}, there exists a further subsequence (which we suppress in the notation) with decomposition $\mu_n^{\otimes p} = \sum_{i \in I} \alpha_{i, n}^{\otimes p} + \beta_n^{\otimes p}$ satisfying the conditions of Lemma~\ref{lem:decomposition-convergence}. Since we know that $\mathbf{D}^{\otimes p}(\widetilde{\mu}_n^{\otimes p} , \widetilde{\mu}^{\otimes p}) \to 0$, it must be that $\{ \widetilde{\alpha}_i^{\otimes p} \} = \{ \widetilde{\mu}^{\otimes p} \}$. In other words, $\mu_n^{\otimes p}= \mu_n'^{\otimes p} + \beta_n^{\otimes p}$ where $\widetilde{\mu}_n'^{\otimes p} \to \widetilde{\mu}^{\otimes p}$ weakly. Since $\mu_n^{\otimes p}$ and $\mu^{\otimes p}$ are both tuples of probability measures, it must be that $\beta_n^j(\mathbb{R}^d) \to 0$ for each $j$, so $\beta_n^{\otimes p} \to 0$ weakly and hence $\widetilde{\mu}_n^{\otimes p} \to \widetilde{\mu}^{\otimes p}$ in the quotient weak topology. Since this is true for all subsequences of $\widetilde{\mu}_n^{\otimes p}$, we have that $\widetilde{\mu}_n^{\otimes p} \to \widetilde{\mu}^{\otimes p}$ in the quotient weak topology.
\end{proof}

\begin{corollary}
\label{cor:conti-functional}
    The functionals $\xi \mapsto \Lambda (p_{\epsilon}^{\otimes p}, \xi^{\otimes p} )$ and $\xi \mapsto \| \xi \ast p_{\epsilon}\|_q$ defined by
    \[
    \Lambda (p_{\epsilon}^{\otimes p}, \xi^{\otimes p} ) := \sum_{\widetilde{\alpha}_i^{\otimes p} \in \xi^{\otimes p}} \int_{\mathbb{R}^{d}} (\alpha_i^1 \ast p_{\epsilon})(x) (\alpha_i^2 \ast p_{\epsilon})(x) \dots (\alpha_i^p \ast p_{\epsilon})(x)  \mathrm{d} x
    \]
    \[
    \| \xi \ast p_{\epsilon} \|_q^q := \sum_{\widetilde{\alpha}_i \in \xi} \| \alpha_i \ast p_{\epsilon} \|_q^q
    \]
    are continuous functions of $\widetilde{\mathcal{X}}^{\otimes p}$ and $\widetilde{\mathcal{X}}$, respectively.
\end{corollary}

\begin{proof}
    For integers $p \ge 2$, consider the test function
    \[
    p_{\epsilon}^{\otimes p}(x^1 , \dots , x^p) = \int_{\mathbb{R}^d} p_\epsilon(x^1 - x) p_\epsilon(x^2 - x) \dots p_\epsilon(x^p - x) \mathrm{d} x.
    \]
    Clearly, $p_{\epsilon}^{\otimes p} \in \mathcal{F}_p$ and so $\xi^{\otimes p} \mapsto \Lambda( p_{\epsilon}^{\otimes p} , \xi^{\otimes p} )$ and $\xi \mapsto\| \xi \ast p_{\epsilon} \|_p^p = \Lambda_p ( p_{\epsilon}^{\otimes p}, \xi)$ are continuous. For real-valued $q > 1$, it suffices to consider sequences $\widetilde{\mu}_n \in \widetilde{\mathcal{M}}_1(\mathbb{R}^d)$ converging to some $\xi \in \widetilde{\mathcal{X}}(\mathbb{R}^d)$. We know that there exists a decomposition $\mu_n = \sum_{i \in I} \alpha_{n ,i} + \beta_n$, where $\widetilde{\alpha}_{n, i} \to \widetilde{\alpha}_i$ in $\widetilde{\mathcal{M}}_1$ and $\beta_n$ totally disintegrates. If we let $p = \lceil q \rceil$, then
    \begin{equation*}
        \| \alpha_{n, i} \ast p_{\epsilon} \|_p \to \| \alpha_i \ast p_{\epsilon} \|_p, \quad \| \alpha_{n, i} \ast p_{\epsilon} \|_1  =\alpha_{n, i}(\mathbb{R}^d) \to \alpha_{i}(\mathbb{R}^d) = \| \alpha_i\ast p_{\epsilon} \|_1 , \quad \| \beta_n \ast p_{\epsilon} \|_p \to 0, \quad \| \beta_n \ast p_{\epsilon} \|_1 \le 1.
    \end{equation*}
    Hence by interpolation, we may conclude that $\| \widetilde{\mu}_n \ast p_{\epsilon} \|_q \to \| \xi \ast p_{\epsilon} \|_q$. Therefore, $\| \cdot \ast p_{\epsilon} \|_q$ is continous for any $q > 1$.
\end{proof}

\subsection{LDP for occupation measures}
\label{subsec:MV-top-ldp}

In this section, we prove Proposition~\ref{prop:brownian-ldp}. We make use of the Mukherjee-Varadhan LDP in $\widetilde{\mathcal{X}}_{\le 1}(\mathbb{R}^d)$ \cite{MukherjeeVaradhan2016} and the Donsker-Varadhan weak LDP in $(\mathcal{M}_1(\mathbb{R}^d))^p$ \cite{DonskerVaradhan1975, DonskerVaradhan1976, DonskerVaradhan1983}. To this end, define the map $\pi^{\otimes p} = (\pi^1 , \dots , \pi^p) : \widetilde{\mathcal{X}}^{\otimes p} \to (\widetilde{\mathcal{X}})^p$ defined by
\[  
\pi(\{ \widetilde{\alpha}_i^{\otimes p} \}_{i \in I}) = (\pi^1(\{ \widetilde{\alpha}^{\otimes p}_i \}) , \dots , \pi^p(\{ \widetilde{\alpha}_i^{\otimes p} \})= (\{ \widetilde{\alpha}_i^1\}_{i \in I}, \dots, \{ \widetilde{\alpha}_i^p\}_{i \in I}) 
\]
In other words, $\pi$ is a projection that forgets the joint diagonal shift of the measures, and maps each coordinate to its own equivalence class in $\widetilde{\mathcal{X}}$ (and deleting all zero measures). We shall also use the shorthand $\xi^j = \pi^j(\xi^{\otimes p})$ to denote each marginal, i.e., $\pi(\xi^{\otimes p}) = (\xi^1 , \dots ,\xi^p)$. Note that for singletons $\widetilde{\mu}^{\otimes p} \in \widetilde{\mathcal{M}}_1^{\otimes p}$, this gives the usual quotient map onto the marginals, $(\widetilde{\mu}^1, \dots , \widetilde{\mu}^p)$.

\begin{lemma}
\label{lem:projection}
    The map $\pi$ is a continuous surjection. Moreover, $\mathcal{I}(\xi^{\otimes p}) = \sum_{j=1}^p \mathcal{I}(\xi^j)$ and $\mathcal{I}(\xi^{\otimes p})$ is lower semicontinuous.
\end{lemma}

\begin{proof}
Surjectivity and $\mathcal{I}(\xi^{\otimes p}) = \sum_{j=1}^p \mathcal{I}(\xi^j)$ are trivial. To prove continuity, it suffices to consider sequences $\widetilde{\mu}_n^{\otimes p} \in \widetilde{\mathcal{M}}_1^{\otimes p}$ converging to some $\xi^{\otimes p} \in \widetilde{\mathcal{X}}_{\le 1}^{\otimes p}$. At this point, we may observe the proof of Corollary~\ref{cor:profile-decomposition} to see that when $\widetilde{\mu}_n^{\otimes p} \to \xi^{\otimes p}$ in $\widetilde{\mathcal{X}}^{\otimes p}$, each component $\widetilde{\mu}_n^j$ also converges to $\xi^j$ in $\widetilde{\mathcal{X}}$. Therefore we may deduce that the maps $\pi^j : \xi \mapsto \xi^j$ are continuous, and thus so is $\pi$. We know from \cite{MukherjeeVaradhan2016} that each $\mathcal{I}(\xi^j)$ is lower semicontinuous, so $\mathcal{I}(\xi^{\otimes p})$ is also lower semicontinuous.
\end{proof}

\begin{remark*}
    This lemma is another reason we chose to divert from the definition of \cite{Mukherjee2017} and \cite{ErhardPoisat2026}. Indeed, with the definition used there, the map $\pi$ does not satisfy any of the three properties (continuity, surjectivity, and $\mathcal{I}(\xi^{\otimes p}) = \sum \mathcal{I}(\xi^j)$) as it loses too much information.
\end{remark*}

\begin{lemma}
    For any closed set $F \subseteq \widetilde{\mathcal{X}}_{\le 1}^{\otimes p}$,
    \[
    \limsup_{t \to \infty} \frac{1}{t} \log \mathbb{P}( \widetilde{L}_t^{\otimes p} \in F) \le - \inf_{ \xi^{\otimes p} \in F} \mathcal{I}(\xi^{\otimes p}).
    \]
\end{lemma}

\begin{proof}
    Since $\widetilde{\mathcal{X}}_{\le 1}^{\otimes p}$ is compact, it suffices to show that for any $\xi^{\otimes p} \in \widetilde{\mathcal{X}}_{\le 1}^{\otimes p}$ and $\epsilon > 0$, there exists some open neighborhood $U_{\epsilon}(\xi^{\otimes p})$ of $\xi^{\otimes p}$ such that
    \[  
    \limsup_{t \to \infty} \frac{1}{t} \log \mathbb{P}(\widetilde{L}_t^{\otimes p} \in U_{\epsilon}(\xi^{\otimes p})) \le -\mathcal{I}(\xi^{\otimes p}) + \epsilon.
    \]
    To this end, define neighborhoods $U^j_{\epsilon}(\xi^j)$ of $\xi^j$ in $\widetilde{\mathcal{X}}_{\le 1}$ such that
    \[
    \limsup_{t \to \infty} \frac{1}{t} \log \mathbb{P}(\widetilde{L}_t^j \in U^j_{\epsilon}(\xi^j)) \le -\mathcal{I}(\xi^j) + \frac{\epsilon}{p}.
    \]
    Such sets exist since we have an LDP for single Brownian motions in $\widetilde{\mathcal{X}}_{\le 1}(\mathbb{R}^d)$ as shown by \cite{MukherjeeVaradhan2016}. Now if we take $U_{\epsilon}(\xi^{\otimes p}) (\pi^{\otimes p})^{-1}(U^1_{\epsilon}(\xi^1) \times \dots \times U^p_{\epsilon}(\xi^p))$ (which is open since $\pi$ is continous), then
    \begin{align*}
        \mathbb{P}(\widetilde{L}_t^{\otimes p} \in U_{\epsilon}(\xi^{\otimes p})) &\le \mathbb{P}(\pi(\widetilde{L}_t^{\otimes p}) \in \pi_{\epsilon}(U(\xi^{\otimes p})) \\
        &= \mathbb{P}( (\widetilde{L}_t^1 , \dots , \widetilde{L}_t^p) \in U^1_{\epsilon} \times \dots \times U^p_{\epsilon}) \\
        &= \mathbb{P} (\widetilde{L}_t^1 \in U^1_{\epsilon}) \times \dots \times \mathbb{P} (\widetilde{L}_t^p \in U^p_{\epsilon})
    \end{align*}
    and hence
    \[  
    \limsup_{t \to \infty} \frac{1}{t} \log \mathbb{P}(\widetilde{L}_t^{\otimes p} \in U_{\epsilon}(\xi^{\otimes p})) \le - \sum_{j=1}^p \mathcal{I}(\xi^j) + \epsilon = -\mathcal{I}(\xi^{\otimes p}) + \epsilon.
    \]
\end{proof}

\begin{lemma}
    For any open set $G \subseteq \widetilde{\mathcal{X}}_{\le 1}^{\otimes p}$,
    \[  
    \liminf_{t \to \infty} \frac{1}{t} \log \mathbb{P}(\widetilde{L}_t^{\otimes p} \in G) \ge - \inf_{\xi^{\otimes p} \in G} \mathcal{I}(\xi^{\otimes p}).
    \]
\end{lemma}

\begin{proof}
    We claim that any $\xi^{\otimes p}$ can be approximated by a sequence $\widetilde{\mu}_n^{\otimes p} \in \widetilde{\mathcal{M}}_1^{\otimes p}$ such that $\widetilde{\mu}_n^{\otimes p} \to \xi^{\otimes p}$ and $\mathcal{I}(\widetilde{\mu}_n^{\otimes p}) \to \mathcal{I}(\xi^{\otimes p})$. Recall the construction at the end of Lemma~\ref{lem:decomposition-convergence}. That is,
    \[  
    \mu_n^j = \sum_{i \in I} \alpha_i^j \ast \delta_{x_i^n} + \beta_n^j,
    \]
    where $\beta_n$ (after normalization) is distributed as a Gaussian with variance $n$. Since $\mathcal{I}(\cdot)$ is subadditive on $\mathcal{M}_{\le 1}$, this gives
    \[
        \mathcal{I}(\mu_n^j) \le \sum_{i \in I} \mathcal{I}(\alpha_i^j) + \mathcal{I}(\beta_n^j) \le \mathcal{I}(\xi^j) + \frac{C}{n}
    \]
    and so $\limsup_{n \to \infty} \mathcal{I}(\widetilde{\mu}_n^{\otimes p}) \le \mathcal{I}(\xi^{\otimes p})$. Combined with the fact that $\mathcal{I}$ is lower semicontinuous, this implies our claim.

    Therefore, we may restrict to $\widetilde{\mathcal{M}}_1^{\otimes p}$ and reduce our lemma to proving
    \[  
    \liminf_{t \to \infty} \frac{1}{t} \log \mathbb{P}(\widetilde{L}_t^{\otimes p} \in G) \ge - \inf \{ \mathcal{I}(\mu_n^{\otimes p}) : \mu_n^{\otimes p} \in (\mathcal{M}_1)^p, \; \widetilde{\mu}_n^{\otimes p} \in G \},
    \]
    which is a direct consequence of the classical Donsker-Varadhan weak LDP on $(\mathcal{M}_1(\mathbb{R}^d)^p$.
\end{proof}

\section{LDP for transformed measures: Proof of Theorems~\ref{thm:self-cond-ldp} - \ref{thm:mutual-tilted-ldp} and Proposition~\ref{prop:intersection-measure-LDP}}
\label{sec:main-proof}

In this section, we prove Theorems~\ref{thm:self-cond-ldp}--\ref{thm:mutual-tilted-ldp} along with Proposition~\ref{prop:intersection-measure-LDP}. Our main tool is the exponential approximation technique described in Section~4.2 of \cite{DemboZeitouni2010}. We state the relevant facts below (rephrased to match our notation) for easy reference.

\begin{definition}[{\cite[Definition~4.2.14]{DemboZeitouni2010}}]
\label{def:exp-approx}
    Let $(\mathcal{Y}, d)$ be a metric space and $Z_t$ a $\mathcal{Y}$-valued random variable. The family $Z_{t, \epsilon}$ of $\mathcal{Y}$-valued variables is an \emph{exponentially good approximation} of $Z_t$ if, for every $\lambda > 0$,
    \[  
    \lim_{\epsilon \to 0} \limsup_{t \to \infty} \frac{1}{t} \log \mathbb{P} (d(Z_t, Z_{t, \epsilon}) > \lambda) = - \infty.
    \]
\end{definition}

\begin{lemma}[{\cite[Theorem~4.2.23]{DemboZeitouni2010}}]
\label{lem:exp-approx}
    Let $\{ \mu_t \}$ be a family of probability measures that satisfy the LDP with a good rate function $\mathcal{I}(\cdot)$ on a Hausdorff topological space $\mathcal{X}$, and for $\epsilon > 0$ let $f_{\epsilon}: \mathcal{X} \to \mathcal{Y}$ be continuous functions, with $(\mathcal{Y}, d)$ a metric space. Assume there exists a measurable map $f : \mathcal{X} \to \mathcal{Y}$ such that for every $\lambda < \infty$,
    \[  
    \limsup_{\epsilon \to 0} \sup_{\{x : \mathcal{I}(x) \le \lambda \}} d(f_{\epsilon} (x), f(x)) = 0.
    \]
    Then any family of probability measures $\mu_t'$ for which $\mu_t \circ f_{\epsilon}^{-1}$ are exponentially good approximations satisfies the LDP in $\mathcal{Y}$ with the good rate function $\mathcal{I}'(y) = \inf \{ \mathcal{I}(x) : y = f(x) \}$.
\end{lemma}

Note that Lemma~\ref{lem:exp-approx} does not depend on the values $f$ takes when $\mathcal{I}(x) = \infty$.

\subsection{LDP for conditional measures}

Now we prove Theorem~\ref{thm:self-cond-ldp} (and Theorem~\ref{thm:mutual-cond-ldp}, which follows a similar scheme). Our strategy is to first lift $\widetilde{L}_t$ into the product space $\widetilde{\mathcal{X}}_{\le 1} \times [0, \infty]$ via the map $Z_t = (\widetilde{L}_t , t^{-1}\| \ell_t \|_q)$. These variables can be approximated with $Z_{t, \epsilon} = (\widetilde{L}_t, t^{-1} \| \ell_{t, \epsilon} \|_q)$. After establishing an LDP for $Z_t$, we may simply restrict to the subset $\widetilde{\mathcal{X}}_{\le 1} \times [1, \infty]$ to prove Theorem~\ref{thm:self-cond-ldp}.

\begin{lemma}
    \label{lem:lift-ldp}
    The distributions of $Z_t = (\widetilde{L}_t , t^{-1}\| \ell_t \|_q)$ satisfy an LDP in $\widetilde{\mathcal{X}}_{\le 1} \times [0, \infty]$ with rate function
    \[  
    \mathcal{I}^{\times}(\xi, y) := \begin{cases}
        \mathcal{I}(\xi) & \text{if } \| \xi \|_q = y \\
        \infty & \text{otherwise}.
    \end{cases}
    \]
\end{lemma}

\begin{proof}
    We define the smoothed approximations $Z_{t, \epsilon} = (\widetilde{L}_t , t^{-1}\| \ell_{t, \epsilon} \|_q)$. Clearly, $Z_{t, \epsilon}$ is the continuous image of $\widetilde{L}_t$ under the map $\xi \mapsto \| \xi \ast p_{\epsilon} \|_q$. Therefore, the contraction principle \cite[Theorem~4.2.1]{DemboZeitouni2010} gives an LDP for $Z_{t, \epsilon}$ with rate function \[  
    \mathcal{I}^{\times}_{\epsilon}(\xi, y) := \begin{cases}
        \mathcal{I}(\xi) & \text{if } \| \xi \ast p_{\epsilon} \|_q = y \\
        \infty & \text{otherwise}.
    \end{cases}
    \]
    Equipped with the product metric $\mathbf{D}^{\times}$ on $\widetilde{\mathcal{X}} \times \mathbb{R}$, Proposition~\ref{prop:self-exp-approx} implies
    \begin{align*}
    \limsup_{\epsilon \to 0} \limsup_{t \to \infty} \frac{1}{t} \log \mathbb{P} (\mathbf{D}^{\times}(Z_t, Z_{t, \epsilon}) > \lambda) &= \limsup_{\epsilon \to 0} \limsup_{t \to \infty} \frac{1}{t} \log \mathbb{P}(\| \ell_t - \ell_{t, \epsilon}\|_q > \lambda t) = - \infty \\
    &\le \limsup_{\epsilon \to 0} \limsup_{t \to \infty} \frac{1}{t} \log \frac{\mathbb{E} \exp \{ M\| \ell_t - \ell_{t, \epsilon} \|_q \}}{e^{M\lambda t}} \\
    &\le \limsup_{\epsilon \to 0} \limsup_{t \to \infty} \frac{1}{t} \log \mathbb{E} \exp \{ M\| \ell_t - \ell_{t, \epsilon} \|_q \} - M\lambda \\
    &= - M \lambda.
    \end{align*}
    Since this holds for all $M > 0$, we can conclude that
    \[  
    \limsup_{\epsilon \to 0} \limsup_{t \to \infty} \frac{1}{t} \log \mathbb{P} (\mathbf{D}^{\times}(Z_t, Z_{t, \epsilon}) > \lambda) = - \infty.
    \]
    Therefore, $Z_{t, \epsilon}$ is an exponentially good approximation of $Z_t$. Now by \eqref{eq:self-rate-approx} below, we have
    \begin{equation}
    \begin{aligned}
    \label{eq:self-rate-bound}
    \limsup_{\epsilon \to 0} \sup_{\mathcal{I}(\xi) \le \lambda} \| \xi - \xi \ast p_{\epsilon } \|_q  &\le \limsup_{\epsilon \to 0} \sup \Bigl\{ \sum_{i \in I} \| \psi_i^2 - \psi_i^2 \ast p_{\epsilon} \|_q : \mathcal{I}(\xi) \le \lambda \Bigr\} \\
    &\le \limsup_{\epsilon \to 0}C \sqrt{\epsilon} \sup \Bigl\{ \sum_{i \in I} (\| \nabla \psi_i \|_2^2 + \| \psi_i \|_2^2) : \mathcal{I}(\xi) \le \lambda \Bigr\} \\
    &\le \limsup_{\epsilon \to 0} C \sqrt{\epsilon} (2\lambda + 1) \\
    &= 0.
    \end{aligned}
    \end{equation}
    In other words, $\widetilde{L}_t$ with the maps $\xi \mapsto (\xi, \| \xi \|_q)$ and $\xi \mapsto \| \xi \ast p_{\epsilon}\|_q$ satisfy the conditions of Lemma~\ref{lem:exp-approx}. Therefore, $Z_t$ satisfies an LDP with good rate function $\mathcal{I}^{\times}(\xi, y)$.
\end{proof} 

\begin{lemma}
\label{lem:rate-approx}
    For any $\psi \in H^1(\mathbb{R}^d)$ and $q > 1$ such that $d(q-1) < 2q$, there exists some $\theta \in (0, 1)$ such that
    \begin{equation}  
    \label{eq:self-rate-approx}
    \| \psi^2 \ast p_{\epsilon} - \psi^2 \|_q \le C \epsilon^{\theta/2} \| \nabla \psi \|_2^{2 - \theta} \| \psi \|_2^{\theta} \le C \epsilon^{\theta/2} ( \| \nabla \psi \|_2^2 + \| \psi \|_2^2).
    \end{equation}
    Similarly, for $\psi^1 , \dots , \psi^p \in H^1(\mathbb{R}^d)$ such that $d(p-1) < 2p$,
    \begin{equation}
        \label{eq:mutual-rate-approx}
        \bigg\| \Big(\prod_{j=1}^p \psi^j\Big)^2 - \Big(\prod_{j=1}^p (\psi^j)^2 \ast p_{\epsilon} \Big) \bigg\|_1 \le C \epsilon^{\theta/2} \Bigl(\sum_{j=1}^p ( \| \nabla \psi^j \|_2^2 + \| \psi^j \|_2^2 ) \Bigr)^p.
    \end{equation}
\end{lemma}

\begin{proof}
    These inequalities are standard corollaries of the Sobolev embedding theorem. We first prove \eqref{eq:self-rate-approx}. Choose some $p > q$ such that there is a continuous embedding $H^1(\mathbb{R}^d) \hookrightarrow L^{2p}(\mathbb{R}^d)$; this is always possible in the regime $d(q-1) < 2q$. For instance, we may take $p = 2q$ when $d = 1, 2$, and $p = 3$ when $d = 3$. By choosing $\theta \in (0, 1)$ such that $q = \theta + (1 - \theta)p$, we may interpolate to get
    \[
    \| \psi^2 \ast p_{\epsilon} - \psi^2 \|_q \le \| \psi^2 \ast p_{\epsilon} - \psi^2 \|_1^{\theta} \| \psi^2 \ast p_{\epsilon} - \psi^2 \|_p^{1 - \theta}.
    \]
    The second term is bounded by the Sobolev embedding theorem,
    \[  
    \| \psi^2 \ast p_{\epsilon} - \psi^2 \|_p \le C \| \psi ^2 \|_p = C \| \psi \|_{2p}^2 \le C \| \nabla \psi \|_2^2.
    \]
    Furthermore,
    \begin{align*}
        \| \psi^2 \ast p_{\epsilon} - \psi^2 \|_1 &\le \int_{\mathbb{R}^d}  p_{\epsilon}(y)\int_{\mathbb{R}^d} |\psi^2(x - y) - \psi^2(x)| \mathrm{d} x \mathrm{d} y \\
        &\le \int_{\mathbb{R}^d} p_{\epsilon}(y) \| \psi^2 \ast \delta_y - \psi^2 \|_1 \mathrm{d} y \\
        &\le C \sqrt{\epsilon} \| \nabla (\psi^2) \|_1 \\
        &= C \sqrt{\epsilon} \| \psi \nabla \psi \|_1 \\
        &\le C \sqrt{\epsilon} \| \psi \|_2 \| \nabla \psi \|_2,
    \end{align*}
    which proves the first inequality. The second is immediate since $\| \nabla \psi \|_2^{2 - \theta} \| \psi \|_2^{\theta} \le \max \{ \| \nabla \psi \|_2^2, \| \psi \|_2^2 \}$. Now to show \eqref{eq:mutual-rate-approx}, simply note that
    \begin{align*}
        \bigg\| \Big(\prod_{j=1}^p \psi^j\Big)^2 - \Big(\prod_{j=1}^p (\psi^j)^2 \ast p_{\epsilon} \Big) \bigg\|_1 &\le \sum_{j_0=1}^p \Bigg\| \Big[ \big( \psi^{j_0})^2 - (\psi^{j_0})^2 \ast p_{\epsilon} \Big] \prod_{j=1} ^{j_0 - 1}(\psi^j)^2 \prod_{j = j_0 + 1}^p \Big[ (\psi^j)^2 \ast p_{\epsilon} \Big] \Bigg\|_1 \\
        &\le \sum_{j_0=1}^p \Big\| \big( \psi^{j_0})^2 - (\psi^{j_0})^2 \ast p_{\epsilon} \Big\|_p \prod_{j=1} ^{j_0 - 1} \big\| (\psi^j)^2 \big\|_p \prod_{j = j_0 + 1}^p \Big\| (\psi^j)^2 \ast p_{\epsilon} \Big\|_p \\
        &\le C \epsilon^{\theta/2}\sum_{j_0=1}^p (\|\nabla \psi^{j_0} \|_2^2 + \| \psi^{j_0} \|_2^2) \prod_{j \ne j_0}  \|\psi^j \|_{2p}^2 \\
        &\le  C \epsilon^{\theta/2}\sum_{j_0=1}^p (\|\nabla \psi^{j_0} \|_2^2 + \| \psi^{j_0} \|_2^2) \prod_{j \ne j_0}  \| \nabla \psi^j \|_2^2,
    \end{align*}
where the third line comes from \eqref{eq:self-rate-approx}.
\end{proof}

Given the above LDP, the proof of Theorems~\ref{thm:self-cond-ldp} is straightforward, as we describe below.

\begin{proof}[Proof of Theorem~\ref{thm:self-cond-ldp}]
Let $A := \{ \| \ell_t \|_q \ge t \} = \{Z_t \in \widetilde{\mathcal{X}}_{\le 1} \times [1, \infty] \}$. Clearly, $\widetilde{\mathcal{X}}_{\le 1} \times [1, \infty]$ has interior $\widetilde{\mathcal{X}}_{\le 1} \times (1, \infty]$. By the contraction principle applied to the projection $(\xi, y) \mapsto y$,
\[  
- \inf \{ \mathcal{I}(\xi): \| \xi \|_q > 1 \} \le \lim_{t \to \infty}  \frac{1}{t} \log \mathbb{P}(\| \ell_t \|_q > 1) \le  \lim_{t \to \infty}  \frac{1}{t} \log \mathbb{P}(\| \ell_t \|_q \ge 1) \le - \inf \{ \mathcal{I}(\xi) : \| \xi \|_q \ge 1 \}.
\]
and both sides converge to $-\Theta_{1, q}$. For any closed set $F \subseteq \widetilde{\mathcal{X}}_{\le 1}$,
\begin{align*}
    \lim_{t \to \infty} \frac{1}{t} \log \mathbb{P}\left(\widetilde{L}_t \in F \; | \; \beta([0, t]^q) \ge t^q \right) &= \lim_{t \to \infty} \frac{1}{t} \log \frac{\mathbb{P}((F \times [0, \infty]) \cap A)}{\mathbb{P}(A)} \\
    &= \lim_{t \to \infty} \frac{1}{t} \left(\log \mathbb{P}( (F \times[1, \infty]) - \log \mathbb{P}(A)\right) \\
    &\le - \inf_{\xi \in F} \left\{ \mathcal{I}(\xi) - \Theta_{1, q}  :  \| \xi \|_q \ge 1 \right\}.
\end{align*}
Similarly for any open $G \subseteq \widetilde{\mathcal{X}}_{\le 1}$,
\begin{align*}
    \lim_{t \to \infty} \frac{1}{t} \log \mathbb{P}\left(\widetilde{L}_t \in G \; | \; \beta([0, t]^q) \ge t^q \right) &\ge \lim_{t \to \infty} \frac{1}{t} \left(\log \mathbb{P}( G \times (1,\infty]) - \log \mathbb{P}(A)\right) \\
    &\ge - \inf_{\xi \in G} \left\{ \mathcal{I}(\xi)- \Theta_{1, q} :  \| \xi \|_q > 1 \right\}.
\end{align*}
We may change $ \| \xi \|_q > 1 $ to $ \| \xi \|_q \ge 1$ since $\| \cdot \|_q$ is continuous on finite sub-level sets of $\mathcal{I}(\cdot)$, and hence the proof is complete.
\end{proof}

\begin{proof}[Proof of Theorem~\ref{thm:mutual-cond-ldp}]
    The proof is almost identical to that of Theorem~\ref{thm:self-cond-ldp} once we replace $(\widetilde{L}_t, t^{-1}\| \ell_t \|_q)$ with $(\widetilde{L}_t^{\otimes p}, t^{-p}\ell_{t, \epsilon}^{\otimes p}(\mathbb{R}^d))$. The only nontrivial part is proving the conditions of Lemma~\ref{lem:exp-approx}. This follows from \eqref{eq:mutual-rate-approx}, namely by

    \begin{equation}
    \label{eq:mutual-rate-bound}
    \begin{aligned}
        \limsup_{\epsilon \to 0} \sup_{\mathcal{I}(\xi^{\otimes p}) \le \lambda} \{ t^{-p}| \ell_t^{\otimes p}(\mathbb{R}^d)- \ell_{t, \epsilon}^{\otimes p}(\mathbb{R}^d)| \} &\le \limsup_{\epsilon \to 0} \sup_{\mathcal{I}(\xi^{\otimes p}) \le \lambda } \biggl\{ \sum_{i \in I} \biggl\| \Bigl(\prod_{j=1}^p \psi^j\Bigr)^2 - \Bigl(\prod_{j=1}^p (\psi^j)^2 \ast p_{\epsilon} \Bigr) \bigg\|_1 \biggr\}\\
        &\le \limsup_{\epsilon \to 0} \sup_{\mathcal{I}(\xi^{\otimes p}) \le \lambda } \biggl\{C \epsilon^{\theta/2} \sum_{i, j} \Bigl(\| \nabla \psi_i^j \|_2^2 + \| \psi_i^j \|_2^2\Bigr)^p \biggr\} \\
        &\le C\epsilon^{\theta/2} (2 \lambda + p)^p.
    \end{aligned}
    \end{equation}
\end{proof}

\subsection{LDP for tilted measures}

Now we prove the LDP for tilted measures, i.e,. Theorems~\ref{thm:self-tilted-ldp} and \ref{thm:mutual-tilted-ldp}. Since the proofs are almost identical, we only present the proof for Theorem~\ref{thm:self-tilted-ldp}. We first strengthen Proposition~\ref{prop:self-exp-approx} into the following lemma.

\begin{lemma}
\label{lem:self-tilt-approx}
    For any $\lambda > 0$ and $1 < \gamma < \frac{2q}{q-1}$,
    \[  
    \limsup_{\epsilon \to 0}\limsup_{t \to \infty} \frac{1}{t} \log \mathbb{E} \exp \Big\{ \lambda t^{1-\gamma} \big|\| \ell_t \|_q^{\gamma} - \| \ell_{t, \epsilon} \|^{\gamma}_q \big| \Big\} = 0.
    \]
\end{lemma}

\begin{proof}
    We wish to prove a moment bound of the form 
    \[  
    \mathbb{E} \Bigl( t^{1-\gamma} \bigl| \| \ell_t \|_q^{\gamma} - \| \ell_{t, \epsilon} \|_q^{\gamma} \bigr| \Bigr)^m \le C^m m! \Bigl( \frac{t^m}{m!} \Bigr)^{\frac{2q - \gamma(q-1)}{2q} - (q-1)\theta}
    \]
    where $C$ may depend on $\gamma$. When $\gamma \le 1$, this is immediate from Corollary~\ref{cor:self-approx-momemt} and the inequality $|a^{\gamma} - b^{\gamma}| \le C |a-b|^{\gamma}$. Now for $\gamma > 1$, note that
    \[  
    \| \ell_t \|_q^{\gamma} - \| \ell_{t, \epsilon} \|_q^{\gamma} \le \gamma (\| \ell_t \|_q - \| \ell_{t, \epsilon}\|_q) (\| \ell_t \|_q + \| \ell_{t, \epsilon}\|_q)^{\gamma - 1}.
    \]
    Therefore, it suffices to bound the moments of the right-hand side. We know from Corollary~\ref{cor:self-approx-momemt} that
    \[  
    \mathbb{E} \| \ell_t - \ell_{t, \epsilon} \|_q^m \le C^m \epsilon^{\theta m} m! \left( \frac{t^m}{m!} \right)^{\frac{q + 1}{2q} - (q-1)\theta}
    \]
    for some small $\theta > 0$. Furthermore, a simple modification of Lemma~\ref{lem:self-approx-moment} by replacing $\Delta_{\epsilon}$ to $\delta$ or $p_{\epsilon}$ yields
    \[  
    \mathbb{E} \| \ell_t \|_q^{m} , \mathbb{E} \| \ell_{t, \epsilon} \|_q^{m} \le C^m m! \left( \frac{t^m}{m!} \right)^{\frac{q+1}{2q}}.
    \]
    More specifically, when $q$ is an integer, one may repeat \eqref{eq:self-approx-riesz} except replacing $| \cdot |^{-2/3}$ with $\delta$ or $p_{\epsilon}$ and $| \cdot |^{-1/3}$ with $|\cdot|^{-1/2}$. Generalizing to fractional $q$ can be done along the lines of Corollary~\ref{cor:self-approx-momemt}. Now by H\"{o}lder's inequality,
     \begin{align*}
         \mathbb{E} t^{(1 - \gamma)m} \big|\| \ell_t \|_q^{\gamma} - \| \ell_{t, \epsilon} \|^{\gamma}_q \big|^m &\le C^m t^{(1 - \gamma)m} \left( \mathbb{E} \| \ell_t - \ell_{t, \epsilon} \|_q^{\gamma m} \right)^{1/\gamma} \left(\mathbb{E} \| \ell_t \|_q^{\gamma m} + \mathbb{E} \| \ell_t \|_q^{\gamma m}  \right)^{(\gamma - 1)/\gamma} \\
         &\le C^m  t^{(1 - \gamma)m}  \times \epsilon^{\theta m} m! \left( \frac{t^m}{m!} \right)^{\frac{q + 1}{2q} - (q-1)\theta} \times (m!)^{\gamma -1} \left( \frac{t^m}{m!} \right)^{(\gamma - 1)\frac{q + 1}{2q}} \\
         &= C^m m! \left( \frac{t^m}{m!} \right)^{\gamma \frac{q + 1}{2q} - (q-1)\theta - \gamma + 1} \\
         &= C^m m! \left( \frac{t^m}{m!} \right)^{\frac{2q - \gamma(q-1)}{2q} - (q-1)\theta}.
     \end{align*}
     By choosing $\theta$ to be sufficiently small, we may assume the exponent on the very right is positive. Hence, we may repeat the proof of Corollary~\ref{cor:sketch-exp-approx} to complete the proof.
\end{proof}

\begin{lemma}
\label{lem:self-tilted-ldp}
    For any closed set $F \subseteq \widetilde{\mathcal{X}}_{\le 1}$,
    \[  
    \limsup_{t \to \infty} \frac{1}{t}\log \int_F \exp \{ t^{1-\gamma} \| \ell_t \|_q^{\gamma}\} \mathrm{d} \mathbb{Q}_t \le \sup_{\xi \in F} \left\{ \| \xi \|_q^{\gamma}- \mathcal{I}(\xi)\right\}.
    \]
    Similarly, for any open set $G \subseteq \widetilde{\mathcal{X}}_{\le 1}$,
    \[  
    \limsup_{t \to \infty} \frac{1}{t}\log \int_G \exp \{ t^{1-\gamma} \| \ell_t \|_q^{\gamma}\} \mathrm{d} \mathbb{Q}_t \ge \sup_{\xi \in G} \left\{ \| \xi \|_q^{\gamma}- \mathcal{I}(\xi)\right\}.
    \]
\end{lemma}

\begin{proof}
    We know that
    \[  
    t^{1-\gamma}  \| \ell_t \|_q^{\gamma}\le t^{1-\gamma} \left| \| \ell_t \|_q^{\gamma} - \| \ell_{t, \epsilon} \|_q^{\gamma} \right|  + t^{1 - \gamma} \| \ell_{t, \epsilon} \|_q^{\gamma} = t^{1-\gamma} \left| \| \ell_t \|_q^{\gamma} - \| \ell_{t, \epsilon} \|_q^{\gamma} \right|  + t\| \widetilde{L}_t \ast p_{\epsilon}\|_q^{\gamma}.
    \]
    By H\"{o}lder's inequality, we have
    \begin{multline*}
    \limsup_{t \to \infty} \frac{1}{t}\log \int_F \exp \{ t^{1-\gamma} \| \ell_t \|_q^{\gamma}\} \mathrm{d} \mathbb{Q}_t \\
    \le \theta\limsup_{t \to \infty} \frac{1}{t}\log \int_F \exp \left\{ \frac{t^{1-\gamma}}{\theta}  \left| \| \ell_t \|_q^{\gamma} - \| \ell_{t, \epsilon} \|_q^{\gamma} \right| \right\} \mathrm{d} \mathbb{Q}_t + (1 - \theta)\limsup_{t \to \infty} \frac{1}{t}\log \int_F \exp \frac{t}{1-\theta}\| \widetilde{L}_t \ast p_{\epsilon}\|_q^{\gamma} \mathrm{d} \mathbb{Q}_t
    \end{multline*}
    for any $\theta \in (0, 1)$. Lemma~\ref{lem:self-tilt-approx} shows that
    \[
    \lim_{\epsilon \to 0} \limsup_{t \to \infty} \frac{1}{t} \log \int_F \exp \left\{ \frac{t^{1-\gamma}}{\theta}  \left| \| \ell_t \|_q^{\gamma} - \| \ell_{t, \epsilon} \|_q^{\gamma} \right| \right\} \mathrm{d} \mathbb{Q}_t \le 0,
    \]
    while Varadhan's lemma implies
    \[
    \limsup_{t \to \infty} \frac{1}{t} \log \int_F \exp \left\{\frac{t}{1-\theta} \| \widetilde{L}_t \ast p_{\epsilon} \|_q^{\gamma} \right\} \mathrm{d} \mathbb{Q}_t \le \sup_{\xi \in F} \left\{ \frac{1}{1-\theta}\| \xi \ast p_{\epsilon} \|_q^{\gamma} - \mathcal{I}(\xi) \right\}.
    \]
    Therefore, by taking $\epsilon \to 0$ followed by $\theta \to 0$, we have
    \begin{align*}
    \limsup_{t \to \infty} \frac{1}{t}\log \int_F \exp \{ t^{1-\gamma} \| \ell_t \|_q^{\gamma}\} \mathrm{d} \mathbb{Q}_t &\le \lim_{\theta \to 0}\limsup_{\epsilon \to 0} \sup_{\xi \in F} \left\{ \frac{1}{1-\theta}\| \xi \ast p_{\epsilon} \|_q^{\gamma} - \mathcal{I}(\xi) \right\} \\
    &= \sup_{\xi \in F} \{ \| \xi \|_q^{\gamma} - \mathcal{I}(\xi) \}.
    \end{align*}
    The last line is justified by \eqref{eq:self-rate-bound}, which shows convergence as $\epsilon \to 0$ for sub-level sets of $\mathcal{I}(\cdot)$.

    The proof for open sets is almost identical, except that we use the inequality
    \[  
    t\| \widetilde{L}_t \ast p_{\epsilon} \|_q^{\gamma} \le t^{1-\gamma}  \| \ell_t \|_q^{\gamma} + t^{1 - \gamma} \left| \| \ell_t \|_q^{\gamma} - \| \ell_{t, \epsilon} \|_q^{\gamma} \right|.
    \]
\end{proof}

\begin{proof}[Proof of Theorem~\ref{thm:self-tilted-ldp}]
For any set $A \subseteq \widetilde{\mathcal{X}}_{\le 1}$, its probability under the Gibbs measure is given by
\[  
\widehat{\mathbb{Q}}_t(A) = \frac{\mathbb{E}^{\mathbb{Q}_t} [\exp t^{1 - \gamma} \|\ell_t \|_q^{\gamma} \mathbf{1}_{\widetilde{L}_t \in A}]}{\mathbb{E}^{\mathbb{Q}_t} [\exp t^{1 - \gamma} \|\ell_t \|_q^{\gamma}]} = \frac{\int_A \exp \{ t^{1 - \gamma} \| \ell_t \|_q^{\gamma} \} \mathrm{d} \mathbb{Q}_t}{\int_{\widetilde{\mathcal{X}}_{\le 1}} \exp \{ t^{1 - \gamma} \| \ell_t \|_q^{\gamma} \} \mathrm{d} \mathbb{Q}_t}.
\]
Since we already have Lemma~\ref{lem:self-tilted-ldp}, the only remaining step is to show that the total mass is given by
\[  
\limsup_{t \to \infty} \frac{1}{t} \log \int_{\widetilde{\mathcal{X}}_{\le 1}} \exp \{ t\| \xi \|_q^{\gamma} \} \mathrm{d} \mathbb{Q}_t = \rho_{1, q, \gamma}.
\]
 By taking $F = G = \widetilde{\mathcal{X}}_{\le 1}$ in Lemma~\ref{lem:self-tilted-ldp}, this is reduced to showing that the supremum
 \[ 
 \rho_{1, q, \gamma} = \sup_{\xi \in \widetilde{\mathcal{X}}_{\le 1}} \{ \| \xi \|_q^{\gamma} - \mathcal{I}(\xi) \}
 \]
 is finite and obtained when $\xi$ is a singleton. We defer this proof to Lemma~\ref{lem:self-tilt-var} of the appendix, where we also show that the solution is unique and given by the optimizer of the Gagliardo-Nirenberg inequality.
\end{proof}

\begin{proof}[Proof of Theorem~\ref{thm:mutual-tilted-ldp}]

For reasons similar to the self-intersecting case, it suffices to show that
\begin{equation}
\label{eq:mutual-tilted-approx}
\limsup_{\epsilon \to 0} \limsup_{t \to \infty} \frac{1}{t} \log \mathbb{E} \exp \{ 
\lambda t^{1-\gamma} |\langle \mathbf{1}, \ell_t^{\otimes p} - \ell_{t, \epsilon}^{\otimes p} \rangle|^{\gamma / p}\}=0.
\end{equation}
To this end, recall the moment bound of Lemma~\ref{lem:mutual-approx-moment}, which implies
\[  
\mathbb{E} t^{(1 - \gamma)m}|\langle f, \ell_t^{\otimes p} - \ell_{t, \epsilon}^{\otimes p} \rangle|^{\gamma m/p} \le C^m \epsilon^{\theta\gamma m/p} \times m! \times \left( \frac{t^m}{m!} \right)^{\frac{\gamma}{p} \left( \frac{2p - d(p-1)}{2} - \theta \right) - \gamma + 1}
\]
where we may take $\theta$ to be arbitrarily small. Since the exponent on the very right simplifies to $\frac{2p - \gamma d(p-1)}{2p} - \frac{\gamma \theta}{p}$, we may choose some $\theta > 0$ so that it is positive for a given $\gamma < \frac{2p}{d(p-1)}$. Hence, we may take $f = \lambda^{p/\gamma}$ and repeat the proof of Corollary~\ref{cor:sketch-exp-approx} to show \eqref{eq:mutual-tilted-approx}.
\end{proof}

\subsection{Proof of Proposition~\ref{prop:intersection-measure-LDP}}

Now we prove Proposition~\ref{prop:intersection-measure-LDP}. As we've already established that $\ell_{t, \epsilon}^{\otimes p}$ is an exponentially good approximation of $\ell_t^{\otimes p}$, the rest is fairly standard. A similar argument was also done in \cite[Section~3]{Mukherjee2017}. Our strategy is to view $t^{-p} \widetilde{\ell}_t^{\otimes p}$ as the image of $\widetilde{L}_t^{\otimes p}$ under the map $\Gamma : \widetilde{\mathcal{X}}^{\otimes p} \to \widetilde{\mathcal{X}}$ defined by
    \[  
    \Gamma(\xi^{\otimes p}) = \begin{cases}
        \{ \widetilde{\gamma}_i \}_{i \in I}, \quad \gamma_i(\mathrm{d} x) = \big( \prod_{j=1}^p (\psi_i^j)^2(x) \big)\mathrm{d} x & \text{if } \mathcal{I}(\xi^{\otimes p}) < \infty \\
        \emptyset & \text{otherwise}.
    \end{cases}
    \]
    Since $\Gamma$ is not continuous, we also define approximations $\Gamma_{\epsilon} : \widetilde{\mathcal{X}}^{\otimes p} \to \widetilde{\mathcal{X}}$ defined by
    \[
    \Gamma_{\epsilon}(\{ \widetilde{\alpha}_i^{\otimes p} \}_{i \in I}) = \{ \widetilde{\gamma}_{i, \epsilon}\}_{i \in I}, \quad \gamma_{i, \epsilon}(\mathrm{d} x) = \Big( \prod_{j=1}^p (\alpha_i^j \ast p_{\epsilon})(x) \Big) \mathrm{d} x.
    \]
    and show that they are exponentially good approximations of $\Gamma$. We remark that it is \emph{not} true that $\Gamma(\widetilde{L}_t^{\otimes p}) = t^{-p} \widetilde{\ell}_t^{\otimes p}$ unless $\mathcal{I}(\widetilde{L}_t^{\otimes p}) < \infty$ (which $\widetilde{L}_t^{\otimes p}$ almost surely is not). However, because $\Gamma_{\epsilon}(\widetilde{L}_t^{\otimes p}) = t^{-p} \widetilde{\ell}_{t, \epsilon}^{\otimes p}$ and we have exponentially good approximations, we can still retrieve the LDP as if $\Gamma(\widetilde{L}_t) = t^{-p} \widetilde{\ell}_{t, \epsilon}^{\otimes p}$ were true everywhere.
\begin{lemma}
\label{lem:smooth-intersection-ldp}
    For any $\epsilon > 0$, the distributions $t^{-p} \widetilde{\ell}_{t, \epsilon}^{\otimes p}$ satisfy an LDP in $\widetilde{\mathcal{X}}^{\otimes p}$ with good rate function
    \[  
    \mathcal{I}_{\epsilon}^{\ell} (\zeta) = \inf\{ \mathcal{I}(\xi^{\otimes p}): \xi^{\otimes p} \in \widetilde{\mathcal{X}}_{\le 1}^{\otimes p} , \; \Gamma_{\epsilon}(\xi^{\otimes p}) = \zeta\}.
    \]
\end{lemma}

\begin{proof}
    Clearly, $\Gamma_{\epsilon}(\widetilde{L}_t^{\otimes p}) = t^{-p} \widetilde{\ell}_{t, \epsilon}^{\otimes p}$. To see that $\Gamma_{\epsilon}$ is continuous, take any $f \in \mathcal{F}_k$ and observe that
    \begin{align*}
    \Lambda_k( f, \Gamma_{\epsilon}(\xi^{\otimes p}) ) &= \sum_{i \in I} \int f(x_1 , \dots ,x_k) \prod_{r=1}^k \Big( \big( \prod_{j=1}^p (\alpha_i^j \ast p_{\epsilon})(x_r) \big) \mathrm{d} x_r \Big)\\
    &= \sum_{i \in I} \int f(x_1, \dots, x_k) \Big(\prod_{r=1}^k \prod_{j=1}^p  p_{\epsilon}(x_r - y_r^j) \Big) \mathrm{d} x_1 \dots \mathrm{d} x_k \mathrm{d} y_1^1 \dots \mathrm{d} y_k^p \\
    &= \Lambda_{(k, \dots , k)}( f_{\epsilon}^{\otimes p}, \xi^{\otimes p}),
    \end{align*}
    where
    \[  
    f^{\otimes p}_{\epsilon} (y_1^1, \dots ,y_1^p, \dots, y_k^p) = \int_{\mathbb{R}^{dk}} f(x_1 , \dots , x_k) p_{\epsilon}(x_1 - y_1^1) \dots p_{\epsilon}(x_1 - y_1^p) \dots p_{\epsilon}(x_k - x_k^p) \mathrm{d} x_1 \dots \mathrm{d} x_k.
    \]
    Since $f_{\epsilon}^{\otimes p}$ is an element of $\mathcal{F}_{kp}$, we can deduce that the maps $\xi^{\otimes p} \mapsto \Lambda_k( f, \Gamma_{\epsilon}(\xi^{\otimes p}) )$ are continuous for every $f \in \mathcal{F}_k$. Therefore, $\Gamma_{\epsilon}$ is also continuous and our claim follows from the contraction principle.
\end{proof}

\begin{proof}[Proof of Proposition~\ref{prop:intersection-measure-LDP}]
    By Lemma~\ref{lem:exp-approx}, it suffices to show that
    \[  
    \limsup_{\epsilon \to 0} \sup_{\mathcal{I}(\xi^{\otimes p}) \le \lambda} \mathbf{D}( \Gamma_{\epsilon}(\xi^{\otimes p}), \Gamma(\xi^{\otimes p})) = 0
    \]
    for any $\lambda < \infty$.

    Recall the proof of Lemma~\ref{lem:smooth-intersection-ldp}. For any $\xi^{\otimes p}$ with $\mathcal{I}(\xi^{\otimes p}) < \infty$, we may write the densities of $\gamma_i - \gamma_{i, \epsilon}$ as
    \[  
    \prod_{j=1}^p (\psi_i^j)^2(x) - \prod_{j=1}^p \bigl((\psi_i^j)^2 \ast p_{\epsilon}\bigr)(x) = \sum_{j_0 = 1}^p \Big( (\psi_i^{j_0})^2 - (\psi_i^{j_0})^2 \ast p_{\epsilon})\Big)(x) \prod_{j < j_0} (\psi_i^j)^2(x) \prod_{j > j_0} (\psi_i^j)^2 \ast p_{\epsilon}(x).
    \]
    Therefore,
    \begin{align*}
        |\Lambda_k(f, \Gamma&(\xi^{\otimes p})) - \Lambda_k(f, \Gamma_{\epsilon} (\xi^{\otimes p}))| \\
        &= \sum_{i \in I}\left| \int f(x_1 , \dots, x_k) \gamma_i (\mathrm{d} x_1) \dots \gamma_i(\mathrm{d} x_k) - \int f(x_1 , \dots, x_k) \gamma_{i, \epsilon} (\mathrm{d} x_1) \dots \gamma_{i, \epsilon}(\mathrm{d} x_k) \right| \\
        &\le \sum_{i \in I}\sum_{r_0 = 1}^k  \left| \int f(x_1 , \dots , x_k) (\gamma_i - \gamma_{i, \epsilon})(\mathrm{d} x_{r_0}) \prod_{r < r_0} \gamma_i(\mathrm{d} x_r) \prod_{r > r_0} \gamma_{i, \epsilon} (\mathrm{d} x_r) \right| \\
        &\le \sum_{i \in I}\sum_{r_0 = 1}^k \bigg\| \int f(x_1 , \dots , x_k) \prod_{r < r_0} \gamma_i(\mathrm{d} x_r) \prod_{r > r_0} \gamma_{i, \epsilon} (\mathrm{d} x_r)\bigg\|_{\infty} \bigg\| \big(\prod_{j=1}^p \psi_i^j\big)^2 - \big(\prod_{j=1}^p (\psi_i^j)^2 \ast p_{\epsilon} \big) \bigg\|_1.
    \end{align*}
    The last line is H\"{o}lder's inequality, where the first term is a function of $x_{r_0}$ and the second term comes from the distribution of $\gamma_i - \gamma_{i, \epsilon}$. We can further bound the first term by $\| f \|_{\sup}$ since each $\psi_i^j$ satisfies $\| \psi_i^j \|_2 \le 1$. The second term is bounded by \eqref{eq:mutual-rate-bound}. Therefore we have
    \[
    |\Lambda_k(f, \Gamma (\xi^{\otimes p})) - \Lambda_k(f, \Gamma_{\epsilon} (\xi^{\otimes p}))| \le  k \| f \|_{\sup} C (2\mathcal{I}(\xi^{\otimes p}) + p)^p \epsilon^{\theta/2}
    \]
    for some small $\theta > 0$. Plugging this into \eqref{eq:MV-metric}, we obtain our desired result.
\end{proof}

\appendix

\section{Weak convergence: Proof of Theorems~\ref{thm:self-cond-conv}--\ref{thm:mutual-tilted-conv}}
\label{sec:variational-problem}

\begin{lemma}[{\cite[Theorem~B]{Weinstein1982/83}}]
\label{lem:gagliardo-nirenberg}
    For any $d$ and $q > 1$ such that $d(q-1) < 2q$, there exists a constant $\kappa_{d, q}$ such that
    \[  
    \| \psi \|_{2q} \le \kappa_{d, q} \| \nabla \psi \|_2^{\frac{d(q-1)}{2q}} \| \psi \|_2^{1 - \frac{d(q-1)}{2q}} \quad \text{for all } \psi \in H^1(\mathbb{R}^d).
    \]
    Moreover, there exists a unique positive, radially symmetric function $\psi_0 \in H^1(\mathbb{R}^d)$ that satisfies the equality with $\| \psi_0 \|_2 = \| \nabla \psi_0 \|_2 = 1$. All other solutions are obtained by the following operations:
    \begin{enumerate}
        \item spatial shifts: $\psi(\cdot -x)$
        \item vertical scaling: $c \psi$
        \item horizontal scaling: $\psi(cx)$.
    \end{enumerate}
\end{lemma}
Note that the two scaling operations can be used to obtain functions $\psi_{a, b}(x) = a \psi(bx)$ which satisfy
\[  
\| \psi_{a, b} \|_2 = ab^{-d/2} \| \psi \|_2, \quad \| \psi_{a, b} \|_{2q} = a b^{-d/2q} \| \psi \|_q, \quad \| \nabla \psi_{a,b} \|_2 = ab^{-d/2 + 1} \| \nabla \psi \|_2.
\]
Hence by altering $a$ and $b$, we can choose an optimal function to the Gagliardo-Nirenberg inequality while choosing two values out of $\| \psi\|_2, \|\psi\|_{2q}, \| \nabla \psi \|_2$.

\begin{lemma}
    \label{lem:self-var-problem}
    The optimization problem
    \[  
    \Theta_{1, q} = \inf_{\xi \in \widetilde{\mathcal{X}}_{\le 1}(\mathbb{R})} \left\{ \mathcal{I}(\xi) : \| \xi \|_q = 1 \right\} = \frac{1}{2}\kappa_{d, q}^{- \frac{2q}{d(q-1)}}
    \]
    has a unique solution which is an element of $\widetilde{\mathcal{M}}_1(\mathbb{R^d})$.
\end{lemma}
\begin{proof}
    Take any $\xi = \{ \widetilde{\alpha}_i \}_{i \in I}$ with $\mathcal{I}(\xi) < \infty$ and denote denote $m_i = \| \psi_i \|_2^2$, $p_i = \| \psi_i \|_{2q}^{2q}$. It is clearly optimal to choose each $\psi_i$ to be solutions to \eqref{eq:GN-inequality} so that $\| \nabla \psi_i \|_2 = (\kappa_{d, q})^{-\frac{2q}{d(q-1)}}\| \psi_i \|_{2q}^{\frac{2q}{d(q-1)}} \| \psi_i \|_2^{1 - \frac{2q}{d(q-1)}}$. Therefore, the variational problem is reduced to solving
\[  
\inf \left\{ \frac{1}{2} \kappa_{d, q}^{-\frac{2q}{d(q-1)}}\sum_{i \in I} p_i^{\frac{2}{d(q-1)}} m_i^{1 - \frac{2q}{d(q-1)}} : \sum_{i \in I} p_i \ge 1, \sum_{i \in I} m_i \le 1 \right\}.
\]
This is bounded by
\begin{align*}
    \sum p_i^{\frac{2}{d(q-1)}} m_i^{1 - \frac{2q}{d(q-1)}} &\ge \left( \sum p_i^{\frac{2}{d(q-1)}} m_i^{1 - \frac{2q}{d(q-1)}} \right) \left( \sum m_i \right)^{\frac{2q - d(q-1)}{d(q-1)}}\ge \left( \sum p_i^{1/q} \right)^{\frac{2q}{d(q-1)}} \ge \left( \sum p_i \right)^{\frac{2}{d(q-1)}} \ge 1.
\end{align*}
The first inequality uses $\sum m_i \le 1$, the second is H\"{o}lder's inequality with weights $\frac{d(q-1)}{2q}$ and $\frac{2q - d(q-1)}{2q}$, the third uses $\sum x_i^q \le (\sum x_i)^q$, and the fourth comes from $\sum p_i \ge 1$. The equality conditions require that $I = \{ i \}$ is a singleton and $m_i = p_i = 1$.
\end{proof}

\begin{lemma}
    \label{lem:self-tilt-var}
    Suppose $d \ge 1$, $q > 1$, and $0 < \gamma$ such that $d(q-1) < 2q$ and $\gamma \le \frac{2q}{d(q-1)}$. Then the variational problem
    \[
    \rho_{d, q, \gamma} = \sup \{ \| \xi \|_q^{\gamma} - \mathcal{I}(\xi) : \xi \in \widetilde{\mathcal{X}}_{\le 1}(\mathbb{R}^d) \} = \left(\frac{2q - \gamma d(q-1)}{2q} \right) \left( \frac{\gamma d(q-1)}{q} \right)^{\frac{\gamma d(q-1)}{2q - \gamma d(q-1)}} \kappa_{d, q}^{\frac{4\gamma q}{2q - \gamma d(q-1)}}
    \]
    has a unique solution, which is an element of $\widetilde{\mathcal{M}}_{1}(\mathbb{R}^d)$.
\end{lemma}

\begin{proof}
Let $\xi \in \tilde{\mathcal{X}}$ such that $\mathcal{I}(\xi) < \infty$ and denote $m_i = \| \psi_i \|_2^2$ and $k_i = \| \nabla \psi_i \|_2^2$. By Lemma~\ref{lem:gagliardo-nirenberg}, $\| \psi_i \|_{2q}^{2q} \le \kappa_{d, q}^{2q} \; k_i^{\frac{d(q-1)}{2}} m_i^{q - \frac{d(q-1)}{2}}$ and there exists functions $\psi_i$ that achieve equality. Therefore, the optimization problem reduces to solving
\[
    \sup \left\{  \left(\kappa_{d, q}^{2q} \sum_{i \in I} k_i^{\frac{d(q-1)}{2}} m_i^{q - \frac{d(q-1)}{2}}\right)^{\gamma/q} - \frac{1}{2}\sum_{i \in I} k_i : \sum_{i \in I} m_i \le 1 \right\}.
\]
This can be further bounded by
\begin{align*}
    \kappa_{d, q}^{2 \gamma} \left(\sum_{i \in I} k_i^{\frac{d(q-1)}{2}} m_i^{q - \frac{d(q-1)}{2}}\right)^{\gamma/q} - \frac{1}{2}\sum_{i \in I} k_i &\le \kappa_{d, q}^{2 \gamma}\left(\sum_{i \in I} \left( k_i^{\frac{d(q-1)}{2}} m_i^{q - \frac{d(q-1)}{2}}\right)^{1/q}\right)^{\gamma} - \frac{1}{2}\sum_{i \in I} k_i \\
    &= \kappa_{d, q}^{2 \gamma}\left(\sum_{i \in I} k_i^{\frac{d(q-1)}{2q}} m_i^{1 - \frac{d(q-1)}{2q}}\right)^{\gamma} -\frac{1}{2} \sum_{i \in I} k_i \\
    &\le \kappa_{d, q}^{2 \gamma}\left( \left(\sum_{i \in I} k_i \right)^{\frac{d(q-1)}{2q}} \left( \sum_{i \in I}m_i \right)^{1 - \frac{d(q-1)}{2q}} \right)^{\gamma} - \frac{1}{2}\sum_{i \in I} k_i \\
    &\le \kappa_{d, q}^{2 \gamma}\left(\sum_{i \in I} k_i \right)^{\frac{\gamma d(q-1)}{2q}} - \frac{1}{2}\sum_{i \in I} k_i \\
    &= \sup_{y \ge 0} \left\{\kappa_{d, q}^{2 \gamma} y^{\frac{\gamma d(q-1)}{2q}} - \frac{1}{2} y\right\} \\
    &= \left(\frac{2q - \gamma d(q-1)}{2q} \right) \left( \frac{\gamma d(q-1)}{q} \right)^{\frac{\gamma d(q-1)}{2q - \gamma d(q-1)}} \kappa_{d, q}^{\frac{4\gamma q}{2q - \gamma d(q-1)}}.    
\end{align*}
The first inequality uses $(\sum x_i)^{1/q} \le \sum x_i^{1/q}$. The third line is H\"{o}lder'is inequality, and the fourth line uses $\sum m_i \le 1$. The last line is simple calculus, and uses the fact that $\gamma d(q-1) < 2q$ to ensure that the supremum is unique. By the equality conditions, the equality holds exactly when $\xi$ is a singleton with $\psi$ satisfying the Gagliardo-Nirenberg equality condition with
\[  
\| \psi \|_2 = 1, \quad \| \nabla \psi \|_2 = \left( \frac{\gamma d(q-1)}{q} \right)^{\frac{2q}{2q - \gamma d(q-1)}} \kappa_{d, q}^{\frac{4\gamma q}{2q - \gamma d(q-1)}}.
\]
Therefore, the problem has a unique maximizer in $\widetilde{\mathcal{M}}_1 \subseteq \widetilde{\mathcal{X}}_{\le 1}$.
\end{proof}

\begin{proof}[Proof of Theorems~\ref{thm:mutual-cond-conv},~\ref{thm:mutual-tilted-conv}]
    By H\"{o}lder's inequality,
    \[  
    \bigg\| \prod_{j=1}^p \psi_i^j \bigg\|_2^2 = \bigg\| \prod_{j=1}^p (\psi_i^j)^2 \Bigg\|_1 \le \prod_{j=1}^p \big\| \psi_i^j \big\|_{2p}^{2p} \le \frac{1}{p} \sum_{j=1}^p \big\| \psi_i^j \big\|_{2p}^{2p}
    \]
    where the equality holds if and only if $\psi_i^1 = \dots = \psi_i^p$. From here, we may repeat the proof of Theorems~\ref{thm:self-cond-conv} and \ref{thm:self-tilted-conv}.
\end{proof}

\bibliographystyle{abbrvurl}
\bibliography{ref}

\end{document}